\documentclass[11pt,reqno]{amsart}
\usepackage[T1]{fontenc}
\usepackage{graphicx}
\usepackage{color}
\definecolor{MyLinkColor}{rgb}{0,0,0.4}
\newcommand{\R}{{\mathbb R}}

\newcommand{\Z}{{\mathbb Z}}
\newcommand{\C}{{\mathbb C}}
\newcommand{\N}{{\mathbb N}}

\newcommand{\s}{\mathbb{S}}

\newcommand{\uu}{\ov{u}}
\newcommand{\zu}{\ov{U}}
\newcommand{\cF}{\mathcal{F}}
\newcommand{\cG}{\mathcal{G}}
\newcommand{\kL}{\mathcal{L}}
\newcommand{\cC}{\mathcal{C}}

\newcommand{\wh}{\widehat}
\newcommand{\wt}{\widetilde}

\newcommand{\ov}{\overline}
\newcommand{\vv}{\mathfrak{v}}
\newcommand{\ww}{\mathfrak{w}}
\newcommand{\ow}{\mathfrak{\overline{w}}}
\newcommand{\uw}{\mathfrak{\underline{w}}}
\newcommand{\zz}{\mathfrak{u}}
\newcommand{\bb}{\mathfrak{a}}
\newcommand{\p}{\partial}

\newcommand{\e}{\varepsilon}

\newcommand{\0}{\Omega}

\newcommand{\spa}{\mathop{\rm span}\nolimits}
\newcommand{\im}{\mathop{\rm Im}\nolimits}
\newcommand{\ke}{\mathop{\rm Ker}\nolimits}

\newtheorem{thm}{Theorem}[section]

\newtheorem{lemma}[thm]{Lemma}
\newtheorem{cor}[thm]{Corollary}
\theoremstyle{remark} 
\newtheorem{rem}[thm]{Remark}

\numberwithin{equation}{section}

\title[Capillary-gravity water waves with piecewise constant vorticity]{Existence of capillary-gravity water waves with piecewise constant vorticity}

\author[C. I. Martin]{Calin Iulian Martin}
\address{Institut f\" ur Mathematik, Universit\" at Wien, Nordbergstra{\ss}e 15,
1090 Wien, Austria}
\email{calin.martin@univie.ac.at}

\author[B.--V. Matioc]{Bogdan--Vasile Matioc}
\address{Institut f\" ur Mathematik, Universit\" at Wien, Nordbergstra{\ss}e 15,
1090 Wien, Austria}
\email{bogdan-vasile.matioc@univie.ac.at}


\usepackage[colorlinks=true,linkcolor=MyLinkColor,citecolor=MyLinkColor]{
hyperref} %dvips

\begin{document}

\begin{abstract}
In this paper we construct  periodic capillarity-gravity water waves with a piecewise constant  vorticity distribution. 
They describe water waves traveling on superposed linearly sheared currents that have  different vorticities.
This is achieved by associating to the height function formulation of the water wave problem a  diffraction problem where we impose suitable
 transmission conditions on each line where the vorticity function has a jump.
The solutions of the diffraction problem, found by using local bifurcation theory, are the desired solutions of the hydrodynamical problem.   
\end{abstract}

\maketitle

\section{Introduction}\label{Sec:1}

We are concerned in this paper with the existence of steady periodic rotational waves interacting with currents  that possess a discontinuous vorticity distribution, 
a situation accounting for
sudden changes in a current and whose numerical simulations have only recently been undertaken, see \cite{KO1,KO2}. 
More precisely, we establish the existence of capillary-gravity waves propagating at constant speed over a flat bed and 
interacting with several vertically superposed and linearly sheared currents of different (constant) vorticities. 
On physical grounds we can justify this situation by the fact that  rotational  waves generated by wind possess a thin layer of high 
vorticity that is adjacent to the wave surface, while
in the near bed region there may exist currents   resulting from sediment transport along the ocean bed.

A rotational fluid is not only interesting as an intricate mathematical problem but also serves a very concrete physical situation since it models wave-current interactions among other
phenomena \cite{Con11, Jon, Thom}.
The waves we consider here are two-dimensional, have an apriori unknown free surface, and the vorticity function is piecewise constant.
Though the vorticity distribution considered in the context of pure gravity waves in \cite{CS11} corresponds to a merely bounded vorticity function, being more general than ours,
we have in addition to gravity also the surface tension as a restoring force.
This has the effect of adding a second order term in the top boundary condition of the height function formulation of the problem,
situation that makes the analysis more intricate.
We enhance that surface tension  appears in the dynamics of water waves in many physical situations one of which is that of wind blowing over a still 
fluid surface and giving rise to two-dimensional small amplitude wave trains driven by capillarity \cite{Ki65} which grow larger and turn into capillary-gravity waves.

In the irrotational regime, the local bifurcation picture was described   in \cite{MJ89} for waves traveling over a fluid layer 
of finite depth, respectively in \cite{JT85, JT86, RS81} for waves of infinite depth.
A particular feature in the irrotational case and  for waves with constant vorticity \cite{CM13x} is  
 that sometimes a mode interacts with another one of half its size giving rise 
to waves  with two crests within a period, so-called  Wilton ripples.
Capillary-gravity water waves with a constant vorticity and stagnation points have been shown to exist in \cite{CM12x} by using technics related to the ones employed 
in \cite{CV11} for pure gravity waves.
Allowing for a general H\"older continuous vorticity distribution,  the existence theory is further developed  in \cite{W06b} for flows without stagnation points.
 It is worth to mention that the many properties of capillary-gravity water waves, such as the regularity of the streamlines
and of the wave profile \cite{Hen10, DH11a, LW12x, AM12x}, or the description of particle trajectories within the fluid \cite{DH07}, were only recently considered
(see \cite{Co06, AC11, CoV, MaA10} for the case when surface tension is neglected). 

In this paper we consider a different context than in \cite{W06b}, namely that of waves with a step function like vorticity.
In order to prove our result we use the height function formulation of the water wave problem which 
 is obtained via the Dubreil-Jacotin transformation (see \cite{CS11} for details) and which has the advantage that the original free boundary problem is rendered into a quasilinear elliptic
problem in a fixed domain.
While in \cite{CS11} the authors worked with a weak formulation of this problem, we are not able to do so 
here. 
This is due to the fact that the top boundary condition is nonlinear and contains second order derivatives of the unknown.
We overcome this difficulty by associating to the height function formulation  a diffraction (or transmission) problem where we impose suitable transmission conditions on each horizontal line where the vorticity has a jump.
Then, using existing results for diffraction problems together with a Fourier multiplier argument, we are able to recast the mathematical problem as an abstract bifurcation problem in a functional analytic context which enables us to use local bifurcation theory. 
One of the difficulties in doing this is due to the lack or rigorous results concerning the $C^{2+\alpha}$-regularity of solutions to diffraction problems close to the interface where transmission conditions are imposed.
We emphasize that diffraction problems are not seldom, they appear when multiphase flows are considered, as is the case of the Muskat problem \cite{EMM12a}.
The solutions that we find solve the boundary conditions of the problem in classical sense and the quasilinear equation in the weak sense defined in \cite{CS11} and almost everywhere in the transformed fluid domain.

In addition to proving existence of waves with the already mentioned properties we derive, in the case of two underlying   currents, the dispersion relation which is an implicit
equation relating: the mean depth, the average thickness of the currents, 
the wavelength,  the constant vorticities of the currents, and 
the  relative speed at the crest.

The outline of the paper is as follows: in  Section \ref{Sec:2} we present the mathematical model and the main result.
Section \ref{Sec:3} is devoted to recasting the equation as an abstract bifurcation problem and to the study of the Fredholm property of an operator associated to a diffraction problem.
In Section \ref{Sec:4} we find necessary conditions for local bifurcation, and in Section \ref{Sec:5} we prove  the main result Theorem \ref{MT} and derive the dispersion relation.

\section{The mathematical model and the main result}\label{Sec:2}

\paragraph{\bf The mathematical model}
We consider herein two-dimensional periodic  waves over a rotational, inviscid, and incompressible fluid, that are  driven by the interplay of gravity and surface tension forces.
Moreover, the waves are assumed to move at constant wave speed $c>0$.
In a reference frame which moves in the same direction as the wave and with  speed $c$, the free surface of the wave is considered to be the graph $y=\eta(x).$

Assuming that the fluid is homogeneous (with unit density), in the fluid domain 
\[
\0_\eta:=\{(x,y)\,:\,\text{$ x\in\s $ and $-d<y<\eta(x)$}\},
\]
 the equations of motion are the steady-state Euler equations
\begin{subequations}\label{eq:P}
  \begin{equation}\label{eq:Euler}
\left\{
\begin{array}{rllll}
(\mathbf{u}-c) {\bf u}_x+{\bf v}{\bf u}_y&=&-{\bf P}_x,\\
({\bf u}-c) {\bf v}_x+{\bf v}{\bf v}_y&=&-{\bf P}_y-g,\\
{\bf u}_x+{\bf v}_y&=&0.
\end{array}
\right.
\end{equation}
We use ${\bf P}$ to denote the dynamic pressure, $({\bf u},{\bf v})$ is the velocity field,    $g$ stands for the gravity constant, and   $d$ is the average depth of the water.
Furthermore, $\s$ is the unit circle, meaning that $\eta, ({\bf u},{\bf v}), {\bf P}$ are $2\pi$-periodic in $x.$
The equations  \eqref{eq:Euler} are supplemented by the following boundary conditions
\begin{equation}\label{eq:BC}
\left\{
\begin{array}{rllll}
{\bf P}&=&{\bf P}_0-\sigma\eta''/(1+\eta'^2)^{3/2}&\text{on $ y=\eta(x)$},\\
{\bf v}&=&({\bf u}-c) \eta'&\text{on $ y=\eta(x)$},\\
{\bf v}&=&0 &\text{on $ y=-d$},
\end{array}
\right.
\end{equation}
with ${\bf P}_0$ denoting the constant atmospheric pressure and $\sigma>0$ being the surface tension coefficient. 
Since the flow is rotational, the vorticity of the flow is the scalar function 
\begin{equation}\label{vor}
\omega:= {\bf u}_y-{\bf v}_x\qquad\text{in $\ov\0_\eta$.}  
\end{equation}
\end{subequations} 

The problem \eqref{eq:P} can also be reformulated as a free boundary value problem, by introducing the stream function $\psi:\ov\0_\eta\to\R$ by the relation
\[
\psi(x,y):=-p_0+\int_{-d}^y({\bf u}(x,s)-c)\, ds\qquad \text{for $(x,y)\in\ov\0_\eta$}.
\]
We have that $\nabla\psi=(-{\bf v},{\bf u}-c)$ and, as shown in \cite{Con11,CoSt04}, the problem \eqref{eq:P} is equivalent to the following system
\begin{equation}\label{eq:psi}
\left\{
\begin{array}{rllll}
\Delta \psi&=&\gamma(-\psi)&\text{in}&\0_\eta,\\
\displaystyle\frac{|\nabla\psi|^2}{2}+g(y+d)-\sigma\frac{\eta''}{(1+\eta'^2)^{3/2}}&=&Q&\text{on} &y=\eta(x),\\
\psi&=&0&\text{on}&y=\eta(x),\\
\psi&=&-p_0&\text{on} &y=-d,
\end{array}
\right.
\end{equation}
the constant $p_0<0$ representing the relative mass flux, and $Q\in\R$ being the so-called total head.   
The function $\gamma,$ called the vorticity function, is obtained by making the additional assumption that the horizontal velocity of each fluid particle is less then the wave speed  
\begin{equation}\label{eq:cond1}
{\bf u}-c<0\qquad\text{in $\ov \0_\eta$,}
\end{equation}
condition valid for waves that are not near breaking \cite{T90}.
This assumption is needed in order to show, cf. \cite{CoSt04, BM11}, that 
\[
\omega(x,y)=\gamma(-\psi(x,y))\qquad\text{in $\ov\0_\eta$.}
\]

The assumption \eqref{eq:cond1} is also crucial when obtaining the third equivalent formulation of the problem \eqref{eq:P}, the height function formulation \eqref{PB}.
Namely, because of \eqref{eq:cond1}, the function
 $\Phi:\ov\0_\eta\to\ov\0$  defined by
\[
\Phi(x,y):=(q,p)(x,y):=(x,-\psi(x,y)),\qquad (x,y)\in\ov\0_\eta,
\]
where $\0:=\s\times(p_0,0),$ 
is a diffeomorphism.
Then, defining the height function   $h:\ov \0\to\R$ by  $h(q,p):=y+d$ for $(q,p)\in\ov\0,$ the problem \eqref{eq:psi}-\eqref{eq:cond1} is equivalent to 
the nonlinear boundary value problem
\begin{equation}\label{PB}
\left\{
\begin{array}{rllll}
(1+h_q^2)h_{pp}-2h_ph_qh_{pq}+h_p^2h_{qq}-\gamma(p)h_p^3&=&0&\text{in $\0$},\\
\displaystyle 1+h_q^2+(2gh-Q)h_p^2-2\sigma \frac{h_p^2h_{qq}}{(1+h_q^2)^{3/2}}&=&0&\text{on $p=0$},\\
h&=&0&\text{on $ p=p_0,$}
\end{array}
\right.
\end{equation}
subjected to the condition 
\begin{equation}\label{PBC}
 \min_{\ov \0}h_p>0.
\end{equation}

\paragraph{\bf The main result}
We construct in this paper solutions of the problem \eqref{PB}-\eqref{PBC} in the case when the vorticity function is a step function.
More precisely, we assume that there exists an integer $N\geq2$,   real constants $\gamma_1, \ldots \gamma_N$ with $\gamma_{i-1}\neq\gamma_i$ for $1\leq i\leq N$, and real numbers $p_0<p_1<\ldots <p_N=0$ such that
\begin{equation}\label{E:GA}
 \gamma(p):=\gamma_i\qquad\text{for $1\leq i\leq N$ and  $p\in(p_{i-1},p_{i})$}.
\end{equation}
At each $p=p_{i}$, $0\leq i\leq N$, the vorticity function may have (has if $1\leq i\leq N-1$) a discontinuity of the first kind.
Our solutions satisfy the boundary conditions  of \eqref{PB} in classical sense and the first equation  almost everywhere in $\0$ and in the following weak sense
\begin{equation}\label{PB1}
 \int_\0\left(\frac{h_q}{h_p}\phi_q-\left(\Gamma+\frac{1+h_q^2}{h_p^2}\right)\phi_p\right) d(q,p)=0\qquad\text{for all $\phi\in C^1_0(\0)$.}
\end{equation}
Hereby, the function $\Gamma$ is the anti-derivative of the vorticity function
\begin{equation}
 \Gamma(p):=\int_0^p\gamma(s)\, ds\qquad\text{for $p_0\leq p\leq0$.}
\end{equation}

Our main result is the following theorem.
\begin{thm}\label{MT}
Let $N\in\N, N\geq 2$, $p_0<p_1<\ldots <p_N=0$, $(\gamma_1,\ldots,\gamma_N)\in\R^N$ be given such that $\gamma_{i-1}\neq\gamma_i$ for $1\leq i\leq N$, 
and let the vorticity function $\gamma$ be defined by \eqref{E:GA}.

Then, given $\alpha\in(0,1),$ there exists a positive integer $n$ and    real-analytic curves $\cC_k,$ $k\in\N\setminus\{0\},$ consisting only of  solutions of problem
 \eqref{PB}-\eqref{PBC} with the property that each solution $h$
 on the curve satisfies $h\in C^{2-}(\ov\0)$, $h(\cdot,0)\in C^\infty(\s),$ and 
 \begin{equation}\label{E:R}
 h\in C^{2+\alpha}(\s\times[p_{i-1},p_{i}])\cap C^{\infty}(\s\times(p_{i-1},p_{i}))\qquad \text{for all $1\leq i\leq N.$}
 \end{equation}
 Each curve $\cC_k$  contains a laminar flow solution (all the streamlines being parallel to the flat bed)
 and all the other points on the curve correspond to solutions that have minimal period $2\pi/(kn) $, only one crest and trough per period, and are symmetric with respect to the crest line. 
 
 These solutions solve the last two equations of \eqref{PB} in  classical sense and the first equation almost everywhere in $\0$  
 (more precisely in $\0\setminus\left(\{\s\times\{p_i\}\,:\, 1\leq i\leq N-1\}\right)$) and in the weak sense defined by \eqref{PB1}.  
\end{thm}

\begin{rem}\label{R:0}
The integer $n$ in Theorem \ref{MT} may be chosen to be $n=1$ provided that the condition \eqref{d2} is satisfied.
\end{rem}

Though our result is true for any arbitrary finite number $N\geq2$, any finite 
sequence $p_0<p_1<\ldots <p_N=0$, and any tuple $(\gamma_1,\ldots,\gamma_N)\in\R^n$, we will prove
first the result when $N=2$ and $(\gamma_1,\gamma_2)\in\R^2$ satisfies $\gamma_1\neq\gamma_2.$
The corresponding result for $\gamma_1=\gamma_2,$ has been already established in 
\cite{W06b} in the context of  waves which satisfy also condition \eqref{PBC}, respectively in \cite{CM12x} 
for waves that possess stagnation points.
For a characterization of the local branches obtained  in these references we refer to \cite{M13x}.
 The proof of Theorem \ref{MT}, with $N\geq3$,    will be discussed after proving the corresponding statement for $N=2$.

 Before of that, let us reconsider the problems \eqref{eq:P} and \eqref{eq:psi} and interpret our solutions in the light of these formulations.
It is well-known that the  streamlines of the flow coincide with the level curves of the stream function $\psi$ and that they are parametrized by the mappings $x\mapsto h(x,p)-d$,
 $p_0\leq p\leq 0$
 (the free wave surface corresponding to $p=0$), cf. \cite{CoSt04}.
 Using this property, we have the following result.
 
 \begin{cor}\label{L:EF}
 Given a solution $h$ of the problem \eqref{PB}-\eqref{PBC} as found  in Theorem \ref{MT}, we define $\eta_i(x):=h(x,p_i)-d$, $0\leq i\leq N,$ and   
 \[
 \0_{\eta}^i:=\{(x,y)\,:\,\eta_{i-1}(x)<y<\eta_{i}(x) \}\qquad  \text{ for $1\leq i\leq N$}.
 \]
 Then, we have   $\eta_i\in C^{2+\alpha}(\s)$ for $1\leq i\leq N-1.$
The other streamlines  are parametrized by smooth functions $h(\cdot,p)-d)$, with $p\in[p_0,0]\setminus\{p_i\,:\, 1\leq i\leq N-1\}$, the wave surface 
being the graph $y=\eta(x):=\eta_N(x)$.
 Furthermore,  the stream function $\psi$ satisfies 
  \begin{align}\label{PG0}
 \psi\in C^{2-}(\ov\0_\eta)\cap C^{2+\alpha}(\ov{\0_{\eta}^i})\cap C^\infty(\0_{\eta}^i)\qquad\text{for all $1\leq i\leq N$},
  \end{align}
 and  we have
 \begin{align}\label{PG}
 {\bf u},{\bf v},{\bf P} \in C^{1-}(\ov\0_\eta)\cap C^{1+\alpha}(\ov{\0_{\eta}^i})\cap C^\infty(\0_{\eta}^i)\qquad\text{for all $1\leq i\leq N$}.
 \end{align}
 \end{cor}
\begin{proof} 
The stream function $\psi$ and the height function $h$ are coupled by the following relation: given $x\in\s,$ the function $\psi(x,\cdot)$
is the solution of the initial value problem 
\begin{align*}
\left\{
\begin{array}{rllll}
 \psi_y(x,y)&=&\displaystyle-\frac{1}{h_p(x,-\psi(x,y))}& \qquad\text{for  $y\leq\eta(x)$},\\
 \psi(x,\eta(x))&=&0,
 \end{array}
 \right.
\end{align*}
cf. \cite{Con11}.
This proves \eqref{PG0}.
The property  \eqref{PG} follows from \eqref{PG0}  and from  Bernoulli's law 
\[
\frac{(c-{\bf u})^2+{\bf v}^2}{2}+gy+{\bf P}+\Gamma(-\psi)=const.\qquad\text{in $\ov\0_\eta.$}
\]
\end{proof}

\section{The associated diffraction problem}\label{Sec:3}

In the following we choose $N=2$, $p_0<p_1<0$, $(\gamma_1, \gamma_2)\in\R^2$ with $\gamma_1\neq\gamma_2$,  and define  the vorticity function $\gamma$  by \eqref{E:GA}.
In order to prove our main result Theorem \ref{MT}, we associate to \eqref{PB} the following diffraction (or transmission) problem
\begin{equation}\label{DP}
\left\{
\begin{array}{rllll}
(1+u_q^2)u_{pp}-2u_pu_qu_{pq}+u_p^2u_{qq}-\gamma_1u_p^3&=&0&\text{in $\0_1$},\\
(1+U_q^2)U_{pp}-2U_pU_qU_{pq}+U_p^2U_{qq}-\gamma_2U_p^3&=&0&\text{in $\0_2$},\\
\displaystyle 1+U_q^2+(2gU-Q)U_p^2-2\sigma \frac{U_p^2U_{qq}}{(1+U_q^2)^{3/2}}&=&0&\text{on $p=0$},\\
u&=&U&\text{on $p=p_1$},\\
u_p&=&U_p&\text{on $p=p_1$},\\
u&=&0&\text{on $p=p_0$},
\end{array}
\right.
\end{equation}
where we used the notation $\0_1:=\s\times (p_{0},p_{1})$ and $\0_2:=\s\times (p_{1},0).$
The reason for defining \eqref{DP} is twofold.
First of all, the nonlinear boundary condition on $p=0$ contains  second order derivatives
of the unknown, so that we cannot  consider a weak formulation of \eqref{PB} in a similar manner as in \cite{CS11}. 
Secondly, using \eqref{PB} we can still work in a classical  H\"older setting and incorporate arbitrarily many jumps of the vorticity function into the problem.
This is due to the fact that each solution of \eqref{DP}  defines a solution of \eqref{PB}.

\begin{lemma}\label{L:1} Assume that $(u,U)\in C^{2+\alpha}(\ov\0_1)\times C^{2+\alpha}(\ov\0_2)$ is a solution of \eqref{DP}. 
Then, the  function $h:\ov \0\to\R$ defined by
\begin{equation}\label{SOL}
 h:=\left\{
 \begin{array}{ll}
  u&\text{in $\ov\0_1$,}\\
  U&\text{in $\ov\0_2$}
 \end{array}
 \right.
\end{equation}
belongs to  $C^{2-}(\ov\0)$ and solves the last two equations of  \eqref{PB} pointwise and the first equation of \eqref{PB}  in $\0_1\cup\0_2.$ 
Moreover,  if $h$ satisfies \eqref{PBC},
then we additionally have
$h\in C^{\infty}(\0_i),$ $i=1,2,$ and $h(\cdot,0)\in C^\infty(\s)$.
\end{lemma}
\begin{proof}
Because of the transmission conditions (equations three and four) of \eqref{DP}, we  easily see that the first order derivatives of $h$ are Lipschitz continuous, that is $h\in C^{2-}(\ov\0_1)$.
Clearly, $h$ solves also the last two equations of  \eqref{PB}  and the first one  in $\0_1\cup\0_2.$
To finish the proof, we note that condition \eqref{PBC} ensures that the first two equations of \eqref{DP} are  uniformly elliptic.
Then, by elliptic regularity, cf. \cite{GT01}, we obtain that $h\in C^{\infty}(\0_i),$ $i=1,2.$
 The property that the wave surface  is smooth, or equivalently that  $h(\cdot,0)\in C^\infty(\s)$, follows  from   Theorem 7.3 in   \cite{ADN54},
by using the uniform ellipticity of the first  equation of \eqref{PB} together with the fact that  the linearization of the boundary condition of \eqref{PB} on $p=0$
 satisfies the complementing condition, see e.g. \cite{HM}.
\end{proof}

In the remainder of this paper we shall seek solutions of \eqref{DP} that satisfy \eqref{PBC}.
According to Lemma \ref{L:1}, they are the solutions described in Theorem \ref{MT}.
To this end, we recast the problem \eqref{DP} as an abstract bifurcation problem, and use then  bifurcation tools to prove the existence of local bifurcation
curves consisting of solutions of \eqref{DP} and \eqref{PBC}.  \smallskip

\paragraph{\bf Laminar flow solutions}
We introduce now an additional parameter $\lambda$ into the problem \eqref{DP} which is   used to parametrize the trivial solutions of \eqref{DP}.
These solutions describe water waves with a flat surface and parallel streamlines, and we call them laminar flow solutions of \eqref{DP}.
The head $Q$ corresponding to a laminar flow solution  also depends  on $\lambda$ and therefore we are left with    only one free parameter in   \eqref{DP}. 

We shall denote the laminar flow solutions by $(\uu,\zu).$
Assuming that $(\uu,\zu)$ depends only on the variable $p$,   we see that $(\uu,\zu)$ solves \eqref{DP} and \eqref{PBC} if  and only if it is a solution of the system
\begin{equation}\label{DPL}
\left\{
\begin{array}{rllll}
\uu''&=&\gamma_1\uu'^3&\text{in $p_0<p<p_1$},\\
\zu''&=&\gamma_2\zu'^3&\text{in $p_1<p<0$},\\
\displaystyle 1+(2g\zu(0)-Q)\zu'^2(0)&=&0,\\
\uu(p_1)&=&\zu(p_1),\\
\uu'(p_1)&=&\zu'(p_1),\\
\uu(p_0)&=&0.
\end{array}
\right.
\end{equation}
Whence, there exists $\lambda>2\max_{[p_0,0]}\Gamma $ such that we have
\begin{equation}\label{E:LFS}
\begin{aligned}
\uu(p)&:=\uu(p;\lambda):=\int_{p_0}^p\frac{1}{\sqrt{\lambda-2\Gamma(s)}}\, ds,\qquad\text{ $p\in[p_0,p_1]$},\\[1ex]
\zu(p)&:=\zu(p;\lambda):=\int_{p_0}^p\frac{1}{\sqrt{\lambda-2\Gamma(s)}}\, ds,\qquad\text{ $p\in[p_1,0]$}.
\end{aligned}
\end{equation}
We observe that   $(\uu,\zu)\in C^\infty([p_0,p_1])\times C^\infty([p_1,0]),$  and that $(\uu,\zu)$  verify the system \eqref{DPL} exactly when
\begin{equation}\label{Q}
Q=Q(\lambda):=\lambda+2g\int_{p_0}^0\frac{1}{\sqrt{\lambda-2\Gamma(p)}}\, dp.
\end{equation}
Let us also observe that the constant $\lambda$ is related to the speed at the top of the laminar flow.
Namely, we have that
\[\sqrt{\lambda}=\frac{1}{\zu_p(0)}=(c-{\bf u})\big|_{y=\eta(x)}.\]

\paragraph{\bf The functional analytic setting}
We define now an abstract  functional analytic setting which allows us to recast  the problem \eqref{DP} as a bifurcation problem.
Therefore, we choose $\alpha\in(0,1)$ and define the Banach spaces:
\begin{align*}
 X&:=\left\{(v,V)\in C^{2+\alpha}(\ov\0_1)\times C^{2+\alpha}(\ov\0_2)\,:\, \text{$v =V $, $v_p =V_p$ on $ p=p_1$,   and $v\big|_{p=p_0}=0$}\right\},\\
 Y&:=C^{\alpha}(\ov\0_1)\times C^{\alpha}(\ov\0_2),\qquad Z:=C^{\alpha}(\s),
\end{align*}
whereby we have identified, when defining $Z$, the unit circle $\s$  with $p=0.$ 
Moreover, we introduce the operator $(\cF,\cG):(2\max_{[p_0,0]}\Gamma,\infty)\times X\to Y\times Z$, with $\cF:=(\cF_1,\cF_2)$, by the following expressions:
\begin{align*}
 \cF_1(\lambda,v):=&(1+v_q^2)(v_{pp}+\uu'')-2(v_p+\uu')v_qv_{pq}+(v_p+\uu')^2v_{qq}-\gamma_1(v_p+\uu')^3,\\
 \cF_2(\lambda,V):=&(1+V_q^2)(V_{pp}+\zu'')-2(V_p+\zu')V_qV_{pq}+(V_p+\zu')^2V_{qq}-\gamma_2(V_p+\zu')^3,\\
  \cG(\lambda,V):=&\displaystyle 1+V_q^2+(2g(V+\zu)-Q(\lambda))(V_p+\zu')^2-2\sigma \frac{(V_p+\zu')^2V_{qq}}{(1+V_q^2)^{3/2}},
\end{align*}
with $Q(\lambda)$   given by \eqref{Q}.
Let us observe that the function $(\cF,\cG)$ is well-defined and that it depends real-analytically on its arguments, that is 
\begin{align}\label{BP0}
 (\cF,\cG)\in C^\omega((2\max_{[p_0,0]}\Gamma,\infty)\times X, Y\times Z).
\end{align}
With this notation, the problem \eqref{DP} is equivalent to the following abstract operator equation
\begin{align}\label{BP}
 (\cF,\cG)(\lambda,(v,V))=0\qquad\text{in $Y\times Z$.}
\end{align}
Since when $(v,V)=0$ the problem \eqref{BP} is equivalent to the system \eqref{DPL}, we have that 
\begin{align}\label{BP1}
 (\cF,\cG)(\lambda,(0,0))=0\qquad\text{for all $\lambda\in(2\max_{[p_0,0]}\Gamma,\infty).$}
\end{align}
Moreover, if $(\lambda,(v,V))$ is a solution of \eqref{BP}, then the function $(u,U):=(\uu+v,\zu+V)$ solves the diffraction problem \eqref{DP} when $Q=Q(\lambda)$ and, according to Lemma \ref{L:1},  it
defines a solution $h$ of the water wave problem \eqref{PB}. 
We remark that  if $(v,V)$ are sufficiently small, the associated solution $h$ of \eqref{PB} satisfies also the condition \eqref{PBC}. \smallskip

\paragraph{\bf The Fredholm property for the linearized operator}
Our main tool in determining non-laminar solutions of  problem \eqref{BP} is the theorem on bifurcations from simple eigenvalues due to Crandall and Rabinowitz, cf. \cite{CR71}.
\begin{thm}\label{CR}
Let $\mathbb{X},\mathbb{Y}$ be real Banach spaces, $I\subset\R$ an  open interval,  and  let the mapping $H\in C^\omega(I \times \mathbb{X},\mathbb{Y})$ satisfy:
\begin{enumerate}
\item[(a)] \label{cr1} $H(\lambda,0)=0$ for all $\lambda\in I$;
\item[(b)] \label{cr2} There exists $\lambda_*\in I$ such that Fr\'echet derivative $\p_x H(\lambda_*,0)$ is a Fredholm operator of index zero with a one-dimensional kernel and 
\[\ke \p_x H(\lambda_*,0)=\{sx_0: s\in\mathbb R, 0\neq x_0\in \mathbb{X}\};\]
\item[(c)] \label{cr3}
 The tranversality condition holds
\[
\p_{\lambda x}H(\lambda_*,0)x_0\not \in \im(\p_x H(\lambda_*,0)).
\]
\end{enumerate}
Then, $(\lambda_*,0)$ is a bifurcation point in the sense that there exists $\epsilon>0$ and a real-analytic curve $(\lambda,x):(-\e,\e)\to I\times \mathbb{X}$  consisting only of solutions of the equation $H(\lambda,x)=0$. 
Moreover, as $s\to 0,$ we have that
\[
\lambda(s)=\lambda_*+O(s) \qquad\text{and}\qquad x(s)=s\chi(s) +O(s^2).
\]
Furthermore there exists an open set $U\subset I\times \mathbb{X}$ with $(\lambda_*,0)\in U$ and 
\[
\{
(\lambda,x)\in U :H(\lambda,x)=0, x\neq 0\}=\{(\lambda(s),x(s)): 0<|s|<\epsilon
\}.
\]
\end{thm}

In order to use Theorem \ref{CR} in the context of the  bifurcation problem \eqref{BP0}-\eqref{BP1},
we need to determine particular $\lambda$ for which  $\p_{(v,V)}(\cF,\cG)(\lambda,0)\in\kL(X,Y\times Z)$ is a Fredholm operator of index zero with a one-dimensional kernel. 
We prove first  that  $\p_{(v,V)}(\cF,\cG)(\lambda,0)$ is a Fredholm operator of index zero for every value of $\lambda\in(2\max_{[p_0,0]}\Gamma,\infty).$ 
The other hypothesis Theorem \ref{CR} are achieved later on by choosing appropriate values for $\lambda$ and by restricting the operator $(\cF,\cG)$
to certain subspaces of $X$ and $Y\times Z.$

It is not difficult to see that the Fr\' echet derivative $\p_{(v,V)}(\cF,\cG)(\lambda,0)$ is the linear operator $(L,T)\in\kL(X,Y\times Z)$, with $L:=(L_1,L_2),$ given by
\begin{equation}\label{L1}
\begin{aligned}
 L_1v:=& v_{pp} +\uu'^2v_{qq}-3\gamma_1\uu'^2v_p,\\
 L_2V:=&V_{pp}+\zu'V_{qq}-3\gamma_2\zu'^2V_p,\\
 TV:=&  2\left[(2g \zu -Q(\lambda))\zu'V_p+g\zu'^2V-\sigma \zu'^2V_{qq}\right]\big|_{p=0},
\end{aligned}
\qquad\text{for $(v,V)\in X$}.
\end{equation}
Showing that $(L,T)$ is a Fredholm operator does not follow from the existing theory on diffraction problems, cf. \cite{LU68}. 
This is due to the fact that when solving the linear diffraction problem $(L,T)(v,V)=((f,F), \varphi)\in Y\times Z,$
there are, as far as we know,  no results that guarantee that $(v,V)$ belong to $X$.
The problem is generated by the transmission conditions on $p=p_1$ which can be used to only show that $v\in C^{2+\alpha}(\s\times [p_0,p_1))\cap C^{1+\alpha}(\ov\0_1)$
and $V\in C^{2+\alpha}(\s\times (p_1,p_0])\cap C^{1+\alpha}(\ov\0_2).$
We will overcome this difficulty by using an  approach based on the elliptic  theory  for linear boundary value problems with Venttsel boundary condition \cite{LT91} 
together with a Fourier multiplier theorem for operators on H\"older spaces of periodic functions \cite{EM09, JL12}.

Before proceeding, we observe that $T$ can be re-expressed by the formula 
\begin{equation}\label{L2}
TV=\frac{2}{\lambda}\left[gV-\lambda^{3/2}V_p- \sigma V_{qq}\right]\big|_{p=0}.
\end{equation}

\begin{thm}\label{T:1}
Assume that $(\gamma_1,\gamma_2)\in\R^2$ and that $\gamma_1\neq\gamma_2.$
Then, for every constant $\lambda\in(2\max_{[p_0,0]}\Gamma,\infty),$    the Fr\' echet derivative  $\p_{(v,V)}(\cF,\cG)(\lambda,0)\in\kL(X,Y\times Z)$ is a Fredholm operator of index zero.
\end{thm}
\begin{proof}
 Let us presuppose that $(L, T_0)\in\kL(X,Y\times Z),$ 
 where 
 \[
 T_0 V:=\left[V - V_{qq}\right]\big|_{p=0}
 \]
 is an isomorphism.
 Then, it is obvious that we also have $(L, 2\sigma\lambda^{-1} T_0)\in{\rm Isom}(X,Y\times Z).$ 
Observing that
 \[
 (L,T)(v,V)=(L, 2\sigma\lambda^{-1} T_0)(v,V)+\left(0, \frac{2}{\lambda}\left[(g-\sigma )V -\lambda^{3/2}V_p\right]\big|_{p=0}\right)
 \]
 for all $(v,V)\in X,$ and since  the operator 
 \[
 \left[V\mapsto \frac{2}{\lambda}\left[(g-\sigma )V -\lambda^{3/2}V_p\right]\big|_{p=0}\right]\in\kL(X,Z)
 \]
 is compact, we deduce that $(L,T)$ is a Fredholm operator of index zero.
 
 Whence, we are left to prove that  $(L, T_0)\in{\rm Isom}(X,Y\times Z)$.
First, we observe that the kernel of the operator $(L, T_0)$ consists only of the zero vector.
Indeed, if $(L,T_0)(v,V)=0$ in $Y\times Z$ and $V$ has a positive maximum  at   $(\ov q,\ov p),$ then $\ov p\in\{p_1,0\},$ and  because $T_0V=0$ we must have  $\ov p=p_1.$ 
But, then also $v$ has a positive maximum  at   $(\ov q,\ov p).$
Applying Hopf's lemma in each domain $\0_1$ and $\0_2$, we  find that $V_p(\ov q,\ov p)<0$ and $v_p(\ov q,\ov p)>0.$ 
This contradicts the transmission condition $v_p=V_p$ on $p=p_1.$ 
 
 It remains to show that $(L, T_0) $ is onto.
 To this end, let $((f,F),\varphi)\in Y\times Z$ be given.
 The results  established in \cite{LT91} for second order elliptic equations with Venttsel boundary conditions imply that there exists a unique solution $W\in C^{2+\alpha}(\ov\0_2)$ of the problem
 \begin{align}\label{VP}
  \left\{
\begin{array}{rllll}
L_2 W&=&F&\text{in $\0_2$},\\
T_0W&=&\varphi&\text{on $p=0$},\\
W&=&0&\text{on $p=p_1$}.
\end{array}
\right.
 \end{align}
This property can be obtained also by referring to \cite{GT01}. 
Indeed, it is not difficult to see that the problem $T_0\wt W=\wt W-\wt W''=\varphi$ possesses a unique solution $\wt W\in C^{2+\alpha}(\s).$
Therefore, the function $W$ solving \eqref{VP} is the solution of the Dirichlet problem 
\[\text{$L_2W=F$\quad \text{in $\0_2$} \qquad $W=\wt\varphi$ \quad \text{on $\p\0_2$},}\]
 whereby $\wt\varphi\in C^{2+\alpha}(\ov\0_2)$ is defined by $\wt\varphi(q,p)=(1-p/p_1)\wt W(q)$.
 Moreover, we introduce the function $w\in C^{2+\alpha}(\ov\0_1)$ as being  the unique solution of the Dirichlet problem
  \begin{align}\label{DVP}
  \left\{
\begin{array}{rllll}
L_1 w&=&f&\text{in $\0_1$},\\
w&=&0&\text{on $p=p_1$},\\
w&=&0&\text{on $p=p_0$}.
\end{array}
\right.
 \end{align}
With this notation, it   suffices to show that  for every $\xi\in C^{1+\alpha}(\s)$, the diffraction problem 
  \begin{align}\label{VPP}
  \left\{
\begin{array}{rllll}
L_1 z&=&0&\text{in $\0_1$},\\
L_2 Z&=&0&\text{in $\0_2$},\\
T_0Z&=&0&\text{on $p=0$},\\
z&=&Z&\text{on $p=p_1$},\\
z_p-Z_p&=&\xi&\text{on $p=p_1$},\\
z&=&0&\text{on $p=p_0$},
\end{array}
\right.
 \end{align}
 possesses a (unique) solution $(z,Z)\in C^{2+\alpha}(\ov\0_1)\times C^{2+\alpha}(\ov\0_2).$
 Indeed, if this is true,  let $(z,Z)$ denote the solution corresponding to $\xi=\left(W_p-w_p\right)\big|_{p=p_1}.$
 Then,  $(v,V):=(w+z, W+Z)$ belongs to $X$ and it solves the equation $(L,T_0)(v,V)=((f,F),\varphi).$ 
 
 Hence, we are left to study the solvability of the problem \eqref{VPP}.
 Using elliptic maximum principles as we did before, it is easy to see that \eqref{VPP} has for each $\xi\in C^{1+\alpha}(\s)$ at most a classical solution. 
 From $T_0Z=0$ on $p=0$ we conclude that in fact  $Z=0$ on $p=0.$
 These facts  and the Theorems 16.1 and 16.2 of \cite{LU68} ensure the existence of a unique solution
 $(z,Z)$ of \eqref{VPP} with $z\in C^{1+\alpha}(\ov\0_1)\cap C^{2+\alpha}(\s\times [p_0,p_1))$ and $Z \in C^{1+\alpha}(\ov\0_2)\cap C^{2+\alpha}(\s\times (p_1,0]).$ 
 Let us remark that we only need to show that the restriction $z|_{p_1}\in C^{2+\alpha}(\s).$
 This property together with the first and last equation of \eqref{VPP} show  that $z\in C^{2+\alpha}(\ov\0_1)$.
 Since $z =Z $ on $p=p_1,$ we   also have $Z\in C^{2+\alpha}(\ov\0_2).$
 To finish the proof, we represent the mapping 
 \begin{equation}\label{FM}
  C^{1+\alpha}(\s)\ni\xi\mapsto z|p_1\in C^{1+\alpha}(\s)
 \end{equation}
 as a Fourier multiplier. 
 To this end, we introduce the functions  $a:=1/\uu'\in C^{\infty}([p_0,p_1])$ and $A:=1/\zu'\in C^{\infty}([p_1,0]),$
 where $(\uu,\zu)$ denote the solutions of \eqref{DPL}.
 Similarly as in \cite{CoSt04}, the first two equations of \eqref{VPP} can be written  in the more concise form
 \begin{equation}\label{L3}
 a^3L_1z=(a^3z_p)_p+(az_q)_q=0\quad \text{in $\0_1$}, \qquad A^3L_2Z=(A^3Z_p)_p+(AZ_q)_q=0 \quad \text{in $\0_2$}.
 \end{equation}
 Considering now the Fourier expansions of $\xi, z,$ and $Z$:
 \[
 \xi(q)=\sum_{k\in\Z}a_ke^{ikq},\quad z(q,p)=\sum_{k\in\Z}z_k(p)e^{ikq}, \quad Z(q,p)=\sum_{k\in\Z}Z_k(p)e^{ikq},
 \]
 we find that the functions $(z_k,Z_k),$ $k\in Z$, solve the following problem:  
   \begin{align}\label{VPPP}
  \left\{
\begin{array}{rllll}
(a^3z_k')'-k^2az_k&=&0&\text{$p_0<p<p_1$},\\
(A^3Z_k')'-k^2AZ_k&=&0&\text{$p_1<p<0$},\\
Z_k(0)&=&0,\\
z_k(p_1)&=&Z_k(p_1),\\
z_k'(p_1)-Z_k'(p_1)&=&a_k,\\
z_k(p_0)&=&0.
\end{array}
\right.
 \end{align}
 We already know from the solvability of \eqref{VPP} that the problem  \eqref{VPPP} possesses for each $k\in \Z$ a unique solution $(z_k,Z_k)$ of regularity $z_k\in C^{ 2+\alpha}([p_0,p_1))\cap  C^{ 1+\alpha}([p_0,p_1])$ and $Z_k\in C^{ 2+\alpha}((p_1,0])\cap  C^{1+\alpha}([p_1,0]).$ 
 Moreover, these functions can be computed explicitly. 
 Indeed, when $\gamma_1\gamma_2\neq0,$ using a substitution similar to that used in \cite[Section 8]{CoSt04}, we find from the first two equations of \eqref{VPPP} that
 \begin{equation}\label{EQ:E}
 z_k=\frac{2\gamma_1}{a}\left(\beta e^{-|k|a/\gamma_1}+\delta  e^{|k|a/\gamma_1}\right)\qquad\text{and}\qquad Z_k=\frac{2\gamma_2}{A}\left(\theta e^{-|k|A/\gamma_2}+\vartheta  e^{|k|A/\gamma_2}\right),
 \end{equation}
with real constants $\beta, \delta,\theta, \vartheta $ that can be determine by solving the last four equations of \eqref{VPPP}.
 After some tedious, though elementary, computations we obtain that $z_k(p_1)=\lambda_k a_k,$ whereby
 \begin{align}\label{Symb}
  \lambda_k:=\frac{a^2(p_1)}{\gamma_1-\gamma_2+a(p_1)\left[\coth\left(\Theta_1|k|\right)+\coth\left(\Theta_2|k|\right)\right]|k|}, \qquad k\in\Z,
 \end{align}
 and $\Theta_1$ and $\Theta_2$ are the positive expressions
 \begin{equation}\label{Expr}
 \Theta_1:=\frac{a(p_0)-a(p_1)}{\gamma_1} \qquad\text{and}\qquad \Theta_2:=\frac{A(p_1)-A(0)}{\gamma_2}.
\end{equation}
 When $k=0,$ the value $\lambda_0$ should be understood as the limit $\lim_{k\to0}\lambda _k.$
 The formula \eqref{Symb} is still true when $\gamma_1=0$ or $\gamma_2=0 $ with the mention that if $\gamma_1=0,$ then we have to replace $\Theta_1$ by 
its limit $\lim_{\gamma_1\to0}\Theta_1$, which is again a positive number (similarly when $\gamma_2=0$). 
 We note that the    solvability of \eqref{VPPP} ensures that the denominator of the right-hand side of \eqref{Symb} has to be different from zero.
 With this observation, it is not difficult to see that
 \begin{align}\label{CCC}
  \sup_{k\in\Z}|k||\lambda_k|<\infty\qquad\text{and}\qquad  \sup_{k\in\Z}|k|^2|\lambda_{k+1}-\lambda_k|<\infty.
 \end{align}
Since the mapping \eqref{FM} can be identified with the Fourier multiplier 
\[
\sum_{k\in\Z}a_ke^{ikq}\mapsto\sum_{k\in\Z}\lambda_ka_ke^{ikq}
\]
we infer from \eqref{CCC} and \cite[Theorem 2.1]{JL12} that it belongs to $\kL(C^{1+\alpha}(\s),C^{2+\alpha}(\s)).$
Consequently, the trace  $z\big|_{p=p_1}\in C^{2+\alpha}(\s)$, and this completes our argument.
\end{proof}

\section{The kernel of $\p_{(v,V)}(\cF,\cG)(\lambda,0)$}\label{Sec:4}

In this section we merely assume that $\gamma_1\in C^\alpha([p_0,p_1])$ and that $\gamma_2\in C^\alpha([p_0,p_1])$. 
From the analysis in Section \ref{Sec:3} it follows that, if  $(v,V)=(v_k(p)\cos(kq), V_k(p)\cos(kq))\in X$ belongs to the kernel of  $\p_{(v,V)}(\cF,\cG)(\lambda,0)$,
then the map 
\begin{equation}\label{eq:DE}
 \vv(p):=
 \left\{
 \begin{array}{lll}
  v_k(p),& p\in[p_0,p_1],\\
  V_k(p),& p\in[p_1,0],
 \end{array}
 \right.
\end{equation}
belongs to the real Hilbert space $H:=\{\vv\in H^2((p_0,0))\,:\, \vv(p_0)=0\}$, and it is also in the kernel of the Sturm-Liouville operator 
$R_{\lambda,\mu}:H\to L_2\times  \R,$ where $L_2:=L_2((p_0,0))$ and 
\begin{equation*}
 R_{\lambda,\mu}\vv:=
 \begin{pmatrix}
  (\bb^3 \vv')'-\mu \bb\vv\\
  (g+\sigma\mu)\vv(0)-\lambda^{3/2}\vv'(0)
 \end{pmatrix},
\end{equation*}
provided that $\mu=k^2.$
Hereby, the function 
\begin{equation}\label{An}
 \bb(p):=\bb(p;\lambda):=\sqrt{\lambda-2\Gamma(p)}, \qquad p\in[p_0,0],
\end{equation}
belongs to $C^{\infty}([p_0,p_1])\cap C^{\infty}([p_1,0])$  for all $\lambda\in (2\max_{[p_0,0]}\Gamma,\infty).$
We note that the first derivative of  $\bb$ has a jump at $p_1.$
Vice versa, if $\vv$ belongs to the kernel of $R_{\lambda,\mu}$ and $\mu=k^2$, then letting $(v_k,V_k)$
 be given by \eqref{eq:DE}, the vector $(v,V)=(v_k(p)\cos(kq), V_k(p)\cos(kq))\in X$ belongs to the kernel of $\p_{(v,V)}(\cF,\cG)(\lambda,0)$.
This correspondence motivates us to study the kernel of  $R_{\lambda,\mu}$.

\begin{lemma}\label{L:2}
For every  $(\lambda,\mu)\in(2\max_{[p_0,0]}\Gamma,\infty)\times[0,\infty)$, the operator $R_{\lambda,\mu}$ is a Fredholm operator of index zero and its kernel is at most one-dimensional.
\end{lemma}
\begin{proof}
 Similarly as in the proof of Theorem \ref{T:1}, we define  the compact perturbation $\mathcal{R}$ of $R_{\lambda,\mu}$ by the relation 
 \[
 \mathcal{R}\vv:=R_{\lambda,\mu}\vv- 
 \begin{pmatrix}
 0\\
  (g+\sigma\mu) \vv(0)
 \end{pmatrix}, \qquad \vv\in H.
 \]
 The first part of our claim follows from the fact that the operator $\mathcal{R}$ is a an isomorphism.
 Indeed, given $(f,z)\in L_2\times\C$, if the vector $\vv\in H $ solves the equation $\mathcal{R} \vv=(f,z),$
 then for all $\varphi\in H^1_0:=\{\varphi\in H^1((p_0,0))\,:\, \varphi(p_0)=0\} $ we have
 \begin{equation}\label{VF}
  \int_{p_0}^0\left(\bb^3\vv'\varphi'+\mu \bb\vv\varphi\right)dp=-z\varphi(0)-\int_{p_0}^0 f\varphi\, dp.
 \end{equation}
The right-hand side of \eqref{VF} defines an element of $\kL(H^1_0,\R),$ while the left-hand side of \eqref{VF} defines, in view of Poincar\'e's inequlality, a bounded coercive bilinear  functional on $H^1_0\times H^1_0.$
Using the Lax-Milgram theorem, cf. \cite[Theorem 5.8]{GT01}, we obtain a unique vector $\vv\in H^1_0$ which solves the variational formulation \eqref{VF}.
It is immediate to see that in fact $\vv\in H.$ 
This proves that indeed $\mathcal{R}\in{\rm Isom}(H, L_2\times\R).$

To finish the proof we see that if $\vv_1,\vv_2\in H$ are two vectors in the kernel of $R_{\lambda,\mu},$
then
\begin{align}\label{C}
 0=((\bb^3 \vv_1')'-\mu \bb\vv_1 )\vv_2-((\bb^3 \vv'_2)'-\mu \bb\vv_2 )\vv_1=((\bb^3 (\vv_2\vv'_1-\vv_1\vv_2'))'\qquad\text{in $(p_0,0)$,}
\end{align}
which implies that $\bb^3 (\vv_2\vv'_1-\vv_1\vv_2')$ is constant.
Since $\vv_1(p_0)=\vv_2(p_0)=0$ and $\bb>0,$ we obtain that $\vv_1$ and $\vv_2$ are linearly dependent.    
\end{proof}

In order to determine when the operator $R_{\lambda,\mu}$ has a nontrivial kernel, we define for each pair  $(\lambda,\mu)\in(2\max_{[p_0,0]}\Gamma,\infty)\times [0,\infty)$ the functions $\zz,\vv\in H^2((p_0,0))$ with
 $\zz:=\zz(\cdot;\lambda,\mu)$ and $\vv:=\vv(\cdot;\lambda,\mu)$ as solutions of  the initial value problems 
\begin{equation}\label{ERU}
\left\{\begin{array}{lll}
  (\bb^3 \zz')'-\mu \bb\zz=0\qquad \text{in $(p_0,0)$},\\[1ex]
  \zz(p_0)=0,\quad \zz'(p_0)=1,
 \end{array}
 \right.\hspace{1cm}
 \left\{\begin{array}{lll}
  (\bb^3 \vv')'-\mu \bb\vv=0\qquad \text{in $(p_0,0)$},\\[1ex]
  \vv(0)=\lambda^{3/2},\quad \vv'(0)=g+\sigma\mu.
 \end{array}
 \right.
\end{equation}
These problems can be seen as system of first order linear ordinary differential equations, and therefore the existence and uniqueness  of  $\zz,\vv$ follows from the classical theory, cf. \cite{A83}.  
\begin{lemma}\label{L:3}
 Given $(\lambda,\mu)\in(2\max_{[p_0,0]}\Gamma,\infty)\times [0,\infty)$, the operator $R_{\lambda,\mu}$ has a nontrivial kernel exactly when  the functions $\zz $ and $\vv$, given by \eqref{ERU}, are linearly dependent. 
\end{lemma}
\begin{proof}
 It is easy to see that if $\zz $ and $\vv$ are linearly dependent, then they  both belong to the kernel of $R_{\lambda,\mu}.$
 On the other hand, if $R_{\lambda,\mu}\mathfrak{z}=0,$ then it follows from relation \eqref{C} that $\{\zz, \mathfrak{z}\}$ and $\{\vv, \mathfrak{z}\}$
 are linearly dependent systems.
\end{proof}

Summarizing, the previous lemmas state that $\ke R_{\lambda,\mu} $ is non-trivial (and has  dimension one) exactly when $(\lambda,\mu)$ is a zero  of the function $\Xi:(2\max_{[p_0,0]}\Gamma,\infty)\times  [0,\infty)\to\R$ given by
\begin{equation}\label{DEF}
\Xi(\lambda,\mu):=\lambda^{3/2}\zz'(0;\lambda,\mu)-(g+\sigma\mu)\zz(0;\lambda,\mu).
\end{equation}
Invoking  \eqref{An} and \eqref{ERU}, we note that the function $\Xi$ is real-analytic.
Particularly, its zeros are isolated.
Of course, we are interested only in the zeros for which $\mu=k^2$ for some $k\in\N.$
If $\mu=0,$ then $\zz$ can be computed explicitly
\[
\zz(p;\lambda,0):=\int_{p_0}^p\frac{\bb^3(p_0)}{\bb^3(s)}\, ds,\qquad p\in[p_0,0].
\]
In this case, $\Xi(\lambda,0)=0$ if and only if $\lambda$ solves the equation
\begin{equation}\label{QU}
 \frac{1}{g}=\int_{p_0}^0\frac{1}{\bb^3(p)}\, dp.
\end{equation}
The right-hand side of \eqref{QU} is a strictly decreasing function of $\lambda,$ 
\[
\int_{p_0}^0\frac{1}{\bb^3(p)}\, dp \, { \underset{\lambda\to\infty}{\longrightarrow } 0}\qquad\text{and}\qquad \int_{p_0}^0\frac{1}{\bb^3(p)}\, dp \, {\underset{\lambda\to2\max_{[p_0,0]}\Gamma}{\longrightarrow} \infty}.
\]
Consequently, there exists a unique $\lambda_0\in(2\max_{[p_0,0]}\Gamma,\infty)$ which satisfies \eqref{QU}.
Since for $\mu>0$ we cannot determine in general an explicit expression for $\zz$, determining the zeros of $\Xi(\cdot,\mu)$ is more difficult.
Nevertheless, we have the following result.

\begin{lemma}\label{L:4}
Let $\lambda_0$ be the unique solution of $\Xi(\lambda,0)=0.$ 
 Then,  we have
 \begin{align}
 (i)\quad &\Xi(\lambda,0)>0 \qquad\text{for all $\lambda>\lambda_0$};\label{LIM1}\\
 (ii)\quad &\lim_{\mu\to\infty} \Xi(\lambda,\mu)=-\infty \qquad\text{for all $\lambda \geq\lambda_0$.}\label{LIM2}
 \end{align}
\end{lemma}

\begin{proof}
 It follows readily from \eqref{ERU} that $\zz$ satisfies the following integral relation
 \begin{align}\label{nnn}
 \zz(p)=\int_{p_0}^p\frac{\bb^3(p_0)}{\bb^3(s)}\, ds+\mu\int_{p_0}^p\frac{1}{\bb^3(s)}\int_{p_0}^s(\bb\zz)(r)\, dr\, ds,\qquad p\in[p_0,0].
 \end{align}
 Therefore, we have
 \begin{align*}
  \Xi(\lambda,0)=g\bb^3(p_0)\left(\frac{1}{g}-\int_{p_0}^0\frac{1}{\bb^3(p)}\, dp\right)>0
 \end{align*}
for all $\lambda>\lambda_0,$ which proves \eqref{LIM1}.

In order to show  \eqref{LIM2}, we fix $\lambda \geq\lambda_0$ and  use \eqref{nnn} to decompose  $\Xi(\lambda,\mu)=T_1+\mu T_2,$
whereby 
\begin{align*}
 T_1&:=\bb^3(p_0)\left(1-(g+\sigma\mu)\int_{p_0}^0\frac{1}{\bb^3(p)}\, dp\right),\\
 T_2&:=\int_{p_0}^0(\bb\zz)(p)\, dp-(g+\sigma\mu)\int_{p_0}^0\frac{1}{\bb^3(s)}\int_{p_0}^s(\bb\zz)(r)\, dr\, ds,
\end{align*}
depend only on   $\mu$.
Since $\bb$ does not depend on $\mu$, we obtain  that $ T_1\to-\infty$ as $\mu\to\infty.$
Before studying the behavior of $T_2$, let us infer  from \eqref{ERU} and \eqref{nnn} that  $\zz$  and $\zz'$ are both positive on $(p_0,0].$  
In fact, an  argument similar to that used to deduce the relations \eqref{Rell} and \eqref{Rel} below shows that  for each  $p\in(p_0,0],$ $\min_{[p,0]}\zz$  and $\min_{[p,0]}\zz'$ grow at an exponential rate as  $\mu\to\infty.$
Thus, proving that $ T_2\to-\infty$ when $\mu\to\infty,$ is not obvious.
However, because $\bb$ does not depend on $\mu,$ integration by parts shows that $ T_2\to-\infty$ when $\mu\to\infty $ if we have
\begin{align*}
 \lim_{\mu\to\infty} \left( \int_{p_0}^0\zz(p)\, dp-\mu^{6/7}\int_{p_0}^0(-p)\zz(p)\, dp\right)=-\infty.
\end{align*}
Noticing that
\begin{align*}
& \left(\int_{p_0}^{-\mu^{-2/3}}\zz(p)\, dp-\mu^{6/7}\int_{p_0}^{-\mu^{-2/3}}(-p)\zz(p)\, dp\right)\\[1ex]
&\leq(\mu^{2/3}-\mu^{6/7}) \int_{p_0}^{-\mu^{-2/3}}(-p)\zz(p)\, dp\underset{\mu\to\infty}\to-\infty,
\end{align*}
it suffices  to prove that
\begin{align}\label{DERE}
 \lim_{\mu\to\infty} \left( \int_{-\mu^{-2/3}}^0\zz(p)\, dp-\mu^{6/7}\int_{-\mu^{-2/3}}^0(-p)\zz(p)\, dp\right)=-\infty.
\end{align}
Because the interval $[-\mu^{-2/3},0]$ is very small when $\mu$ is large,  
we can approximate $\zz$ on this interval by solutions of some linear initial value problems with constant coefficients.
To be more precise, we let $\ww$ denote the solution of the linear initial value problem
\begin{equation}\label{CC}
\left\{\begin{array}{lll}
  \ww''+ C\ww'-\mu D\ww=0\qquad \text{in $(p_0,0)$},\\[1ex]
  \ww(-\mu^{-2/3})=A,\quad \ww'(-\mu^{-2/3})=B,
 \end{array}
 \right.
\end{equation}
where $C\in\R$ and $D>0$ are constants, $A:=\zz(-\mu^{-2/3}),$ and $B:=\zz'(-\mu^{-2/3}).$
The solution $\ww$ of \eqref{CC} is given by the following formula
\begin{align}\label{Rell}
 \ww(p)=\frac{A}{r_1-r_2}\left(r_1e^{r_2(p+\mu^{-2/3})}-r_2e^{r_1(p+\mu^{-2/3})}\right)+\frac{B}{r_1-r_2}\left(e^{r_1(p+\mu^{-2/3})}-e^{r_2(p+\mu^{-2/3})}\right)
\end{align}
for $p\in [-\mu^{-2/3},0]$, whereby
\begin{equation}\label{ov}
r_1:=\frac{-C+\sqrt{C^2+4 D\mu}}{2}\qquad\text{and}\qquad r_2:=\frac{-C-\sqrt{C^2+4 D \mu}}{2}.
\end{equation}
The idea of considering the problem \eqref{CC} is the following: defining the $\mu$-dependent functions 
\begin{align}
\underline C:=\max_{[-\mu^{-2/3},0]} \frac{3\bb'}{\bb},\quad \overline C:=\min_{[-\mu^{-2/3},0]} \frac{3\bb'}{\bb},\quad \underline D:=\min_{[-\mu^{-2/3},0]} \frac{1}{\bb^2},\quad \overline D:=\max_{[-\mu^{-2/3},0]} \frac{1}{\bb^2},
\end{align}
it is not difficult to see by subtracting the equations of \eqref{CC} from those satisfied by $\zz$ that we have
\begin{align}\label{Rel}
 \uw\leq \zz\leq \ow \qquad\text{ on $[-\mu^{-2/3},0]$.}
\end{align}
Hereby, $\uw$ and $\ow$ are the solutions of \eqref{CC} corresponding to   $(\underline C, \underline D),$ and $(\overline C, \overline D),$ respectively.
Therefore, the relation \eqref{DERE} is fulfilled if we  show that
\begin{align}\label{DEREK}
 \lim_{\mu\to\infty} \left( \int_{-\mu^{-2/3}}^0\ow(p)\, dp-\mu^{6/7}\int_{-\mu^{-2/3}}^0(-p)\uw(p)\, dp\right)=-\infty.
\end{align}
An elementary computation now gives
\begin{align*}
  \int_{-\mu^{-2/3}}^0\ow(p)\, dp-\mu^{6/7}\int_{-\mu^{-2/3}}^0(-p)\uw(p)\, dp=AT_A+BT_B,
\end{align*}
whereby
\begin{align*}
 T_A:=&\frac{1}{\overline r_1-\overline r_2}\left(\frac{\overline r_1}{\overline r_2}(e^{\overline r_2/\mu^{2/3}}-1)-\frac{\overline r_2}{\overline r_1}(e^{\overline r_1/\mu^{2/3}}-1)\right)
 -\frac{\mu^{6/7-2/3}}{\underline r_1-\underline r_2}\left(\frac{\underline r_2}{\underline r_1}-\frac{\underline r_1}{\underline r_2}\right)\\
& -\frac{\mu^{6/7}}{\underline r_1-\underline r_2}\left(\frac{\underline r_1}{\underline r_2^2}(e^{\underline r_2/\mu^{2/3}}-1)-\frac{\underline r_2}{\underline r_1^2}(e^{\underline r_1/\mu^{2/3}}-1)\right),\\
T_B:=&\frac{1}{\overline r_1-\overline r_2}\left(\frac{1}{\overline r_1}(e^{\overline r_1/\mu^{2/3}}-1)-\frac{1}{\overline r_2}(e^{\overline r_2/\mu^{2/3}}-1)\right)
 -\frac{\mu^{6/7-2/3}}{\underline r_1-\underline r_2}\left(\frac{1}{\underline r_2}-\frac{1}{\underline r_1}\right)\\
& -\frac{\mu^{6/7}}{\underline r_1-\underline r_2}\left(\frac{1}{\underline r_1^2}(e^{\underline r_1/\mu^{2/3}}-1)-\frac{1}{\underline r_1^2}(e^{\underline r_2/\mu^{2/3}}-1)\right),\\
\end{align*}
and with $(\underline r_1, \underline r_2)$ and $(\overline r_1, \overline r_2)$ being defined by \eqref{ov} with $(C,D)$ being replaced by $(\underline C, \underline D),$ and $(\overline C, \overline D),$ respectively.
The claim \eqref{DEREK} follows from the following properties
\begin{align}\label{FP}
 T_A\to_{\mu \to\infty}-\infty\qquad\text{and}\qquad \mu^{2/3 }T_B\to_{\mu\to\infty}-\infty.
\end{align}

We establish first the claim for $T_A.$
Clearly, it  suffices to show that
\begin{align*}
 & \frac{\underline r_1-\underline r_2}{\overline r_1-\overline r_2}\left(\frac{\overline r_1}{\overline r_2}(e^{\overline r_2/\mu^{2/3}}-1)-
\frac{\overline r_2}{\overline r_1}(e^{\overline r_1/\mu^{2/3}}-1)\right)
 - \mu^{6/7-2/3}\left(\frac{\underline r_2}{\underline r_1}-\frac{\underline r_1}{\underline r_2}\right)\nonumber\\
&-\mu^{6/7}\left(\frac{\underline r_1}{\underline r_2^2}(e^{\underline r_2/\mu^{2/3}}-1)-\frac{\underline r_2}{\underline r_1^2}(e^{\underline r_1/\mu^{2/3}}-1)\right)
=:E_1-E_2-E_3\to_{\mu\to\infty}-\infty.
\end{align*}
Let us now observe that $(\underline C, \underline D)$ and $(\overline C, \overline D)$ converge, when $\mu\to\infty$, 
towards the constant pair  $(C,D):=(3\bb'(0)/\bb(0), 1/\bb^2(0))$.
Because $\overline r_i/\mu^{2/3}\to0$   as $\mu\to\infty,$   we easily see that $ E_1\to0.$ 
Moreover, there exists a constant  $K$, independent of $\mu,$ such that
\begin{align*}
 \left| E_2\right|&=\mu^{6/7-2/3} \left|\frac{\underline r_2}{\underline r_1}-\frac{\underline r_1}{\underline r_2}\right|=\mu^{6/7-2/3} 
 \frac{\left|\underline r_2^2-\underline r_1^2\right|}{\underline r_1\underline r_2}\leq K\mu^{6/7-2/3-1/2} 
\end{align*}
and we are left to consider $E_3.$
To this end, we write $E_3=E_{3a}+E_{3b}$ where we set
\[
E_{3a}:=\mu^{6/7} \left(\frac{\sqrt{D \mu }}{D \mu}(e^{-\sqrt{D}\mu^{-1/6}}-1)+\frac{\sqrt{D\mu}}{ D\mu}(e^{\sqrt{D}\mu^{-1/6}}-1)\right).
\]
We note that $E_{3a}$ is obtained by replacing in the definition of $E_3$ the constants $(\underline C, \underline D)$ and $(\overline C, \overline D)$ by 
their limit  $(C,D)$ and retaining only the highest order terms in $\mu.$
Since by the mean value theorem we have 
\[
\max\{|\overline D-D|,|\underline D-D|\}\leq K \mu^{-2/3},
\]
where $K$ is again independent of  $\mu,$
 one can show that $  E_{3b}\to_{\mu\to\infty}0.$
 Furthermore, using  l'Hopitals rule we get
 \begin{align*}
  \lim_{\mu\to\infty}E_{3a} =\frac{7}{15}\lim_{\mu\to\infty}\frac{e^{ \sqrt{D} \mu^{-1/6}}- e^{-\sqrt{D}\mu^{-1/6}}}{\mu^{-1/6}}\mu^{1/42}=\infty.
 \end{align*}
Thus, $  E_{3a}\to\infty,$ and we conclude that $ T_A\to_{\mu\to\infty}-\infty.$
The second claim of \eqref{FP} follows similarly. 
This finishes the proof.
\end{proof}

Invoking Lemma \ref{L:4}, we find for every $\lambda\geq\lambda_0$ a unique constant $\mu(\lambda)\in[0,\infty)$ such that 
\begin{equation}\label{dm}
\begin{aligned}
& \Xi(\lambda,\mu(\lambda))=0;\\
&\text{$\Xi(\lambda,\mu)<0$ for all $\mu>\mu(\lambda).$}
\end{aligned}
\end{equation}

Since   $\Xi(\lambda,0)>0$ when $\lambda>\lambda_0,$  any  zero of $\Xi(\lambda,\cdot)$ satisfies $\mu>0$, provided $\lambda>\lambda_0.$
We next prove  that $\lambda\mapsto \mu(\lambda) $ is a real-analytic and strictly increasing curve.

\begin{lemma}\label{L:5} Assume that $(\ov\lambda,\ov\mu)\in[\lambda_0,\infty)\times(0,\infty)$ satisfies $\Xi(\ov\lambda,\ov\mu)=0.$
Then, we have
\begin{align}\label{slim}
 \Xi_\lambda(\ov\lambda,\ov\mu)>0\qquad\text{and}\qquad \Xi_\mu(\ov\lambda,\ov\mu)<0.
\end{align}
\end{lemma}

Before proving the lemma, let us observe that since $\Xi(\lambda,0)>0$ for all $\lambda>\lambda_0,$  the second relation of  \eqref{slim} ensures   additionally to \eqref{dm} that
\begin{equation}\label{dm2}
\begin{aligned}
&\text{$\Xi(\lambda,\mu)>0$ for all $\mu\in[0,\mu(\lambda)).$}
\end{aligned}
\end{equation}
Particularly, if $\lambda>\lambda_0,$ then  $\Xi(\lambda,\mu)=0$ if and only if $\mu=\mu(\lambda).$
 
\begin{proof}[Proof of Lemma \ref{L:5}]
 Because $\Xi(\ov\lambda,\ov\mu)=0, $ it follows from the Lemmas \ref{L:2}-\ref{L:3}, and the  discussion   following them, that 
 $\ke R_{\ov\lambda,\ov\mu} $  is spanned by the function $\zz$ defined by \eqref{ERU} when $(\lambda,\mu)=(\ov\lambda,\ov\mu)$.
 Thus, $\zz$ solves the following system of equations 
 \begin{equation}\label{bb}
\left\{\begin{array}{lll}
  (\bb^3 \zz')'-\ov\mu \bb\zz=0\qquad \text{in $(p_0,0)$},\\[1ex]
  \zz(p_0)=0,\quad \zz'(p_0)=1,\\[1ex]
  (g+\sigma\mu)\zz(0)-\lambda^{3/2}\zz'(0)=0.
 \end{array}
 \right.
\end{equation}
Differentiating the equations of \eqref{ERU} with respect to $\mu$ shows that the Fr\'echet derivative $\zz_\mu:=\zz_\mu(\cdot,\ov\lambda,\ov\mu)$ is the solution of the problem
 \begin{equation}\label{bb1}
\left\{\begin{array}{lll}
  (\bb^3 \zz_\mu')'-\ov\mu \bb\zz_\mu=\bb\zz\qquad \text{in $(p_0,0)$},\\[1ex]
  \zz_\mu(p_0)=0,\quad \zz'_\mu(p_0)=0.
 \end{array}
 \right.
\end{equation}
First, we establish that
\begin{align}\label{De1}
 \Xi_\mu(\ov\lambda,\ov\mu)=\ov\lambda^{3/2}\zz_\mu'(0)-\sigma\zz(0)-(g+\sigma\ov\mu)\zz_\mu(0)<0.
\end{align}
Multiplying the first equation of \eqref{bb} by $\zz_\mu$ and the first equation of \eqref{bb1} by $\zz,$ we obtain after integrating by parts that
\begin{equation}\label{fc1}
\zz(0)\left(\ov\lambda^{3/2}\zz_\mu'(0)-(g+\sigma\ov\mu)\zz_\mu(0)\right)=\int_{p_0}^0\bb\zz^2\, dp.
\end{equation}
Moreover, if we multiply the first equation of \eqref{bb} by $\zz$ and integrate it by parts we find  that
\[
\int_{p_0}^0\bb^3\zz'^2\, dp-g\zz^2(0)=\ov\mu \left(\sigma\zz^2(0)-\int_{p_0}^0\bb\zz^2\, dp\right).
\]
But, since $\lambda\geq\lambda_0,$ 
 \begin{align*}
  g\zz^2(0)=&g\left(\int_{p_0}^0(\bb^{3/2}\zz')\frac{1}{\bb^{3/2}}\, dp\right)^2\leq g\int_{p_0}^0\bb^{3}\zz'^2\, dp\int_{p_0}^0\frac{1}{\bb^{3}}\, dp\leq \int_{p_0}^0\bb^{3}\zz'^2\, dp,
 \end{align*}
cf. \eqref{QU}, which implies that the right-hand side of \eqref{fc1} is bounded from above by $\sigma\zz^2(0).$
This proves \eqref{De1}.

For the first claim of \eqref{De1} we note that 
\begin{align*}
 \Xi_\lambda(\ov\lambda,\ov\mu)=\ov\lambda^{3/2}\zz_\lambda'(0)+\frac{3}{2}\ov\lambda^{1/2}\zz'(0)-(g+\sigma\ov\mu)\zz_\lambda(0),
\end{align*}
whereby $\zz_\lambda$ is the solution of 
\begin{equation}\label{bb2}
\left\{\begin{array}{lll}
  (\bb^3 \zz_\lambda')'-\ov \mu \bb\zz_\lambda=-(3\bb^2\bb_\lambda\zz')+\ov\mu\bb_\lambda\zz\qquad \text{in $(p_0,0)$},\\[1ex]
  \zz_\lambda(p_0)=0,\quad \zz'_\lambda(p_0)=0,
 \end{array}
 \right.
\end{equation}
and $\bb_\lambda=1/(2\bb).$
Similarly as before, we multiply the first equation of \eqref{bb} by $\zz_\lambda$ and the first equation of \eqref{bb2} by $\zz$ to obtain, after integrating by parts, that
\[
\Xi_\lambda(\ov\lambda,\ov\mu)=\frac{1}{\zz(0)}\int_{p_0}^0\left(\frac{3\bb}{2}\zz'^2+\frac{\ov\mu}{2\bb}\zz^2\right)\, dp>0.
\]
This completes our argument.
 \end{proof}
 
 \begin{rem}\label{R:1}
 The restriction $\mu>0$ in Lemma \ref{L:5} was needed just to prove the second claim of \eqref{slim}.
 But, when  $\mu=0$ and $\lambda_0,$ we obtain easily from \eqref{De1}, \eqref{fc1}, \eqref{QU}, and the explicit expression for $\zz(\cdot\lambda_0,0)$
 that $\Xi_\mu(\lambda_0,0)\leq0$ if and only if
 \begin{equation}\label{d2}
   \int_{p_0}^0\bb(p)\left(\int_{p_0}^p\frac{1}{\bb^3(s)}\, ds\right)^2\, dp\leq\frac{\sigma}{g^2}.
 \end{equation}
 \end{rem}

Combining the previous lemmas, we obtain the following result.
\begin{lemma}\label{L:6} The function 
 \[
 [\lambda_0,\infty)\ni\lambda\mapsto\mu(\lambda)\in(0,\infty)
 \]
 is continuous,  real-analytic in $(\lambda_0,\infty)$,   and strictly increasing.
\end{lemma}
\begin{proof}
 Let $\lambda_1>\lambda_0.$ 
 Because of $\Xi(\lambda_1,\mu(\lambda_1))=0$ and  $\Xi_\mu(\lambda_1,\mu(\lambda_1))<0$ there exists a real-analytic function $\ov\mu$ such that $\ov\mu(\lambda_1)=\mu(\lambda_1)$ 
 and $\Xi(\lambda,\ov\mu(\lambda))=0$ for all $ \lambda$ close to $\lambda_1.$
  Recalling  \eqref{slim}, we see that the function $\ov\mu$ is strictly increasing, as we have
 \[
 \Xi_\lambda(\lambda,\ov\mu(\lambda))+ \Xi_\mu(\lambda,\ov\mu(\lambda)) \ov\mu'(\lambda)=0
 \]
The conclusion follows now from the relations \eqref{dm}-\eqref{dm2}. 
\end{proof}

The next lemma ensures that the function $\mu:[\lambda_0,\infty)\to[\mu(\lambda_0),\infty)$ is bijective.

\begin{lemma}\label{L:7}
 We have that
 \begin{equation}\label{mil}
  \lim_{\lambda\to\infty}\frac{\mu(\lambda)}{\lambda}=\infty.
 \end{equation}
\end{lemma}
\begin{proof}
 Assume by contradiction that there exists a sequence $\lambda_n\to\infty $ and a constant $K>0$  such that $0<\mu(\lambda_n)/\lambda_n\leq K$ for all $n\in\N.$
 For every $n\in\N,$ we denote by $\zz_n$ the function that spans the kernel of $R_{\lambda_n,\mu(\lambda_n)} $ and solves the first system of \eqref{ERU}.
 Then it follows readily  form \eqref{nnn} that there exists a constant $\wt K$, independent of $n$, such that
 \[
 0\leq \zz_n(p)\leq \wt K\left(1+\int_{p_0}^p\zz_n(s)\, ds\right)\qquad\text{for all $p\in[p_0,0]$ and $n\in\N.$}
 \]
 Using Gronwall's inequality, we conclude that the sequence $(\zz_n)_n\subset C([p_0,0])$ is bounded. 
 Since by \eqref{DEF}, \eqref{nnn}, and our assumption we have
 \[
 \Xi(\lambda_n,\mu(\lambda_n))\geq \bb^3(p_0)-(g+\sigma\mu(\lambda_n))\zz_n(0)\to_{n\to\infty}\infty,
 \]
 we obtain a contradiction with the properties defining the map $\mu(\cdot),$ cf. \eqref{dm}. 
\end{proof}

\section{Proof of the main result}\label{Sec:5}

\paragraph{\bf The case $N=2$.}
We  now come back to the setting presented in Theorem \ref{MT} and assume that $N=2$.
We summarize  from the Lemmas \ref{L:4}-\ref{L:7} that there exists a smallest integer
 $n\in\N\setminus\{0\}$ such that for all $k\in\N\setminus\{0\},$ there exists a unique constant $ \lambda_k\in(\lambda_0,\infty)$ with the property that 
\begin{align}\label{LP}
 \mu(\lambda_k):=(kn)^2.
\end{align}
Because $\lambda_k>\lambda_0,$ $\lambda_0$ being the unique solution of \eqref{QU}, Lemma \ref{L:6} ensures that the sequence $(\lambda_k)_k$ is strictly increasing. 
Whence, we have that $F(\lambda_k,(ln)^2)= 0$ if and only if  $k=l.$
Moreover, it follows from the  Remark \ref{R:1} and \eqref{dm}-\eqref{dm2} that the integer $n$  can be chosen to be $n=1$ if \eqref{d2} is satisfied.
This is due to the fact that  $\mu(\lambda_0)=0$ if \eqref{d2} holds true. 

Since we want to determine nontrivial solutions of the water wave problem \eqref{PB} we  study the existence of bifurcation branches consisting of solutions of \eqref{BP} that arise from $(\lambda_k,0),$ with $k\geq1.$
In the following we denote by $\wh X, \wh Y,$ and  $\wh Z$ the subspaces of $X, Y,$ and $ Z,$ respectively, that consist only of $(2\pi/(kn))$-periodic  and even functions in the variable $q$.
Then, it follows readily from the definition of $(\cF,\cG)$ that we have
\begin{align}\label{BP0'}
 (\cF,\cG)\in C^\omega((2\max_{[p_0,0]}\Gamma,\infty)\times \wh X, \wh Y\times \wh Z).
\end{align}
Moreover, the arguments used in the proof of Theorem \ref{T:1} show that the Fr\' echet derivative $\p_{(v,V)}(\cF,\cG)(\lambda,0)\in\kL(\wh X, \wh Y\times \wh Z)$
is a Fredholm operator of index zero for every  value $\lambda\in (2\max_{[p_0,0]}\Gamma,\infty)$.
Taking into account that the kernel of $\p_{(v,V)}(\cF,\cG)(\lambda,0)$  is finite dimensional, 
by the choice of the sequence $(\lambda_k)_k,$ we know that $\p_{(v,V)}(\cF,\cG)(\lambda_k,0)$ has a one-dimension kernel.
More precisely, 
\[
\ke\p_{(v,V)}(\cF,\cG)(\lambda_k,0)=\spa\{(v^0,V^0):=(v_k^0(p)\cos(knq), V_k^0(p)\cos(knq))\}
\]
whereby $(v_k^0,V_k^0)$ defines, cf. \eqref{eq:DE}, a vector $\vv^0\in H^2((p_0,0))$ that spans the one-dimensional kernel of  the operator $R_{\lambda_k,(kn)^2}$.
This vector $\vv^0$ is colinear to the solutions of  both initial value problems \eqref{ERU}.

In order to apply the bifurcation theorem of Crandall and Rabinowitz to the equation $F(\lambda,(v,V))=0$ in $\wh Y\times\wh Z$, which will give us, via Lemma \ref{L:1},
 the desired result from Theorem \ref{MT}, we are left check that
\begin{align}\label{TC}
 \p_{\lambda (v,V)}(\cF,\cG)(\lambda_k,0)(v^0,V^0)\notin \im \p_{(v,V)}(\cF,\cG)(\lambda_k,0)
\end{align}
if $0\neq (v^0,V^0)\in \ke\p_{(v,V)}(\cF,\cG)(\lambda_k,0)$.
To this end, we need to   characterize the range $\im \p_{(v,V)}(\cF,\cG)(\lambda_k,0).$
\begin{lemma}\label{L:R} Given   $k\geq1$, the vector $((f,F), \varphi)\in\wh Y\times\wh Z$ belongs to $ \im \p_{(v,V)}(\cF,\cG)(\lambda_k,0)$ if and only if  we have
 \begin{equation}\label{RN}
  \int_{\0_1}a^3v^0f\, d(q,p)+  \int_{\0_2}A^3V^0F\, d(q,p)+\int_{\s\times\{0\}}\frac{A^2V^0\varphi}{2}\, dq=0
 \end{equation}
 for all $(v^0,V^0)\in \ke\p_{(v,V)}(\cF,\cG)(\lambda_k,0).$
\end{lemma} 
\begin{proof}
We pick  $0\neq (v^0,V^0)\in \ke\p_{(v,V)}(\cF,\cG)(\lambda_k,0)$, and    presuppose  that there exists a pair $(v,V)\in\wh X$  with the property that $ \p_{(v,V)}(\cF,\cG)(\lambda_k,0)(v,V)=((f,F), \varphi)$, that is
 \begin{equation} \label{VVV} 
 L_1v=f\quad \text{in $\wh Y_1$},\qquad L_2V=F\quad \text{in $\wh Y_2$},\qquad TV=\varphi\quad \text{in $\wh Z$},
 \end{equation}
 whereby $L_1, L_2, T$ are given by \eqref{L1} and \eqref{L2}.
 Invoking \eqref{L3}, we use integration and the fact that   $(v^0,V^0)\in\wh X$ solves \eqref{VVV} when  $((f,F), \varphi)=0, $  to find 
 \begin{align*}
  &\int_{\0_1}a^3v^0f\, d(q,p)+  \int_{\0_2}A^3V^0F\, d(q,p)+\int_{\s\times\{0\}}\frac{A^2V^0\varphi}{2}\, dq\\
  &=-\int_{\0_1}(a^3vv_p^0+av_qv_q^0)\, d(q,p)-\int_{\0_2}(A^3VV_p^0+AV_qV_q^0)\, d(q,p)\\
  &\phantom{= \, \,}+\int_{\s\times\{0\}}(g+\sigma(kn)^2)V^0V\, dq=0,
 \end{align*}
the last equality being obtain by using once more the fact that $(v^0,V^0)\in\wh X$ solves \eqref{VVV} when  $((f,F), \varphi)=0.$
Taking into account that the relation \eqref{RN} defines a closed subspace of $ \wh Y\times\wh Z$ which    has codimension one and contains the image $\im \p_{(v,V)}(\cF,\cG)(\lambda_k,0)$, which has itself codimension one,
we obtain the desired claim.
\end{proof}

\begin{lemma}\label{L:TC} Let $k\in\N$ with $k\geq 1$ be given.
The transversality condition \eqref{TC} holds true for all  vectors $0\neq (v^0,V^0)\in \ke\p_{(v,V)}(\cF,\cG)(\lambda_k,0)$.
\end{lemma}
\begin{proof}
 Let $0\neq (v^0,V^0)\in \ke\p_{(v,V)}(\cF,\cG)(\lambda_k,0)$ be given.
 Then, we infer from \eqref{L1}, \eqref{L2}, and  \eqref{L3}, that
 \begin{align*}
  \p_{\lambda (v,V)}\cF_1(\lambda_k,0)(v^0,V^0)&=-\frac{2a_\lambda}{a^3}v^0_{qq}+\frac{3aa'_\lambda-3a_\lambda a'}{a^2}v^0_p=:f,\\
  \p_{\lambda (v,V)}\cF_2(\lambda_k,0)(v^0,V^0)&=-\frac{2A_\lambda}{A^3}V^0_{qq}+\frac{3AA'_\lambda-3A_\lambda A'}{A^2}V^0_p=:F,\\
  \p_{\lambda (v,V)}\cG(\lambda_k,0)(v^0,V^0)&=-\left[\frac{3}{A}V^0_p\right]\big|_{p=0}=:\varphi.
 \end{align*}
Our claim is equivalent to showing that $((f,F), \varphi)$ does not satisfy \eqref{RN}. 
Observing that  $a_\lambda= 1/(2a) $ and $A_\lambda= 1/(2A)$,  we compute  
 \begin{align*}
  &\int_{\0_1}a^3v^0f\, d(q,p)+  \int_{\0_2}A^3V^0F\, d(q,p)+\int_{\s\times\{0\}}\frac{A^2V^0\varphi}{2}\, dq\\
  &=-\int_{\0_1}\left(a^{-1}v^0_{qq}v^0+ 3 a'v_p^0v^0\right)\, d(q,p)-\int_{\0_2}\left(A^{-1}V^0_{qq}V^0+ 3 A'V_p^0V^0\right)\, d(q,p)\\
  &\phantom{= \, \,}-\int_{\s\times\{0\}}\frac{3A}{2}V_p^0V^0\, dq.
 \end{align*}
 On the other hand, using the fact that  $(v^0,V^0)\in \ke\p_{(v,V)}(\cF,\cG)(\lambda_k,0) $  and integration by parts, we obtain that
  \begin{align*}
 &\int_{\0_1} a'v_p^0v^0\, d(q,p)+\int_{\0_2} A'V_p^0V^0\, d(q,p)\\
 &=-\int_{\s\times\{0\}}\frac{A}{2}V_p^0V^0\, dq+\int_{\0_1} \left(\frac{a}{2}(v_p^0)^2+\frac{1}{2a} (v_q^0)^2\right)\, d(q,p)\\
 &\phantom{= \, \,}+\int_{\0_2}\left(\frac{A}{2}(V_p^0)^2+\frac{1}{2A} (V_q^0)^2\right)\, d(q,p),
 \end{align*}
 which   together with the previous
relation gives 
 \begin{align*}
  &\int_{\0_1}a^3v^0f\, d(q,p)+  \int_{\0_2}A^3V^0F\, d(q,p)+\int_{\s\times\{0\}}\frac{A^2V^0\varphi}{2}\, dq\\
 &=-\int_{\0_1} \left(\frac{3a}{2}(v_p^0)^2+\frac{1}{2a}(v_q^0)^2\right)\, d(q,p)-\int_{\0_2}\left(\frac{3A}{2}(V_p^0)^2+\frac{1}{2A}(V_q^0)^2\right)\, d(q,p)<0.
 \end{align*}
 This proves the lemma.
\end{proof}

\paragraph{\bf The case $N\geq3$.}
When the vorticity function has several jumps, we proceed as in Section \ref{Sec:3}  and associate to problem \eqref{PB} a diffraction problem
\begin{equation}\label{DP'}
\left\{
\begin{array}{rllll}
(1+u_{i,q}^2)u_{i,pp}-2u_{i,p}u_{i,q}u_{i,pq}+u_{i,p}^2u_{i,qq}-\gamma_iu_{i,p}^3&=&0&\text{in $\0_i$,  \small{$i\in\{1,N\}$}},\\
\displaystyle 1+u_{N,q}^2+(2gu_N-Q)u_{N,p}^2-2\sigma \frac{u_{N,p}^2u_{N,qq}}{(1+u_{N,q}^2)^{3/2}}&=&0&\text{on $p=0$},\\
u_i&=&u_{i+1}&\text{on $p=p_i$, \small{$i\in\{1,N-1\}$}},\\
u_{i,p}&=&u_{i+1,p}&\text{on $p=p_i$,  \small{$i\in\{1,N-1\}$}},\\
u_1&=&0&\text{on $p=p_0$},
\end{array}
\right.
\end{equation}
where $\0_i:=\s\times(p_{i-1}, p_i)$ for all $1\leq i\leq N$.
Then, similarly to Lemma \ref{L:1} we see that each solution of \eqref{DP} defines a solution of problem \eqref{PB}.
For $N\geq 3$, our main result  is obtained by following the lines of the proof when $N=2$ with the evident modifications.  
There is only one point where the analysis is different, namely when showing that the Fr\'echet derivative
$\p_{(v_1,\ldots,v_N)}(\cF,\cG)(\lambda,0)\in \kL(X,Y\times Z)$
\footnote{For $N\geq 3$ it is natural to define:
\begin{align*}
 X&:=\left\{(v_1,\ldots,v_N)\in\Pi_{i=1}^N C^{2+\alpha}(\ov\0_i)\,:\, \text{$v_i=v_{i+1}$, $v_{i,p}=v_{i+1,p}$ on $ p=p_i$, $i\in\{1,N-1\}$, $0=v_1\big|_{p=p_0}$}\right\},\\
 Y&:=\Pi_{i=1}^N C^{\alpha}(\ov\0_i),\quad Z:=C^{\alpha}(\s).
\end{align*}
}  
is a Fredholm operator for all $\lambda\in(2\max_{[p_0,0]}\Gamma,\infty)$.
As in  the proof of Theorem \ref{T:1} one can show that $\p_{(v_1,\ldots,v_N)}(\cF,\cG)(\lambda,0)$ is a Fredholm operator of index zero provided that the unique solution $(z_1,\ldots, z_N)$
of the diffraction  problem
\begin{equation}\label{DP2}
\left\{
\begin{array}{rllll}
(a_i^3z_{i,p})_p+(a_iz_{i,q})_q&=&0&\text{in $\0_i$,  \small{$i\in\{1,N\}$}},\\
T_0z_{N}&=&0&\text{on $p=0$},\\
z_i&=&z_{i+1}&\text{on $p=p_i$, \small{$i\in\{1,N-1\}$}},\\
z_{i,p}&=&z_{i+1,p}&\text{on $p=p_i$,  \small{$i\in\{1,N-1\}\setminus\{l\}$}},\\
z_{l,p}-z_{l+1,p}&=&\varphi&\text{on $p=p_l$,}\\
z_1&=&0&\text{on $p=p_0$}.
\end{array}
\right.
\end{equation}
belongs to $\Pi_{i=1}^{N}C^{2+\alpha}(\ov\0_i)$ for all $1\leq l\leq N-1$ and all $\varphi\in C^{1+\alpha}(\s).$
Hereby, we set $a_i:=\sqrt{\lambda-2\Gamma}\in C^\infty([p_{i-1},p_{i}]) $ for all $1\leq i\leq N.$ 
Existence of a solution $(z_1,\ldots, z_N)$ of the linear diffraction problem \eqref{DP2} in the class 
\[
\left(C^{2+\alpha}(\s\times[p_0,p_1))\times\Pi_{i=2}^{N-1}C^{2+\alpha}(\0_i)\times C^{2+\alpha}(\s\times(p_{N-1},0])\right)\cap \Pi_{i=1}^{N}C^{1+\alpha}(\ov\0_i)
\]
is obtained by using elliptic maximum principles and the results of \cite{LU68}.
The problem lies in ensuring  $C^{2+\alpha} $ regularity at the interfaces where we have transmission conditions, an argument as in Theorem \ref{T:1} being impossible because $N$ is arbitrary.
Nevertheless we can use the result established  in Theorem \ref{T:1}.
Indeed, if $1\leq l\leq N-1$, then $(z_l,z_{l+1})$ solves the diffraction problem 
\begin{equation*}
\left\{
\begin{array}{rllll}
(a_l^3z_{l,p})_p+(a_lz_{l,q})_q&=&0&\text{in $\s\times((p_{l-1}+p_{l})/2,p_l), $}\\
(a_{l+1}^3z_{l+1,p})_p+(a_{l+1}z_{l+1,q})_q&=&0&\text{in $\s\times(p_{l},(p_{l}+p_{l+1})/2), $}\\
z_{l}&=&\phi_l&\text{on $p=(p_{l-1}+p_{l})/2$},\\
z_{l+1}&=&\phi_{l+1}&\text{on $p=(p_{l}+p_{l+1})/2$},\\
z_{l}&=&z_{l+1}&\text{on $p=p_l$},\\
z_{l,p}-z_{l+1,p}&=&\varphi_l&\text{on $p=p_l$,}
\end{array}
\right.
\end{equation*}
whereby $\varphi_l\in C^{1+\alpha}(\s)$ (possibly $\varphi_l=0$) and  
 \[
 \text{$\phi_l:=z_{l}\big|_{p=(p_{l-1}+p_{l})/2}$, $\phi_{l+1}:=z_{l+1}\big|_{p=(p_{l}+p_{l+1})/2}$ belong to $C^{2+\alpha}(\s)$}.
 \]
Then, we can write
\[
(z_l,z_{l+1})=(v_l,v_{l+1})+(w_l,w_{l+1}),
\]
whereby $v_l\in C^{2+\alpha}(\s\times[(p_{l-1}+p_{l})/2,p_l])$  solves the Dirichlet problem
\begin{equation*}
\left\{
\begin{array}{rllll}
(a_l^3v_{l,p})_p+(a_lv_{l,q})_q&=&0&\text{in $\s\times((p_{l-1}+p_{l})/2,p_l), $}\\
v_{l}&=&\phi_l&\text{on $p=(p_{l-1}+p_{l})/2$},\\
v_{l}&=&0&\text{on $p=p_{l}$},
\end{array}
\right.
\end{equation*}
the function  $v_{l+1}\in C^{2+\alpha}(\s\times[p_l,(p_{l}+p_{l+1})/2])$ is the solution of
\begin{equation*}
\left\{
\begin{array}{rllll}
(a_{l+1}^3v_{l+1,p})_p+(a_{l+1}v_{l+1,q})_q&=&0&\text{in $\s\times(p_{l},(p_{l}+p_{l+1})/2), $}\\
v_{l+1}&=&\phi_{l+1}&\text{on $p=(p_{l}+p_{l+1})/2$},\\
v_{l+1}&=&0&\text{on $p=p_l$,}
\end{array}
\right.
\end{equation*} 
and $(w_l,w_{l+1})$ solves the diffraction problem
\begin{equation*}
\left\{
\begin{array}{rllll}
(a_l^3w_{l,p})_p+(a_lw_{l,q})_q&=&0&\text{in $\s\times((p_{l-1}+p_{l})/2,p_l), $}\\
(a_{l+1}^3w_{l+1,p})_p+(a_{l+1}w_{l+1,q})_q&=&0&\text{in $\s\times(p_{l},(p_{l}+p_{l+1})/2), $}\\
w_{l}&=&0&\text{on $p=(p_{l-1}+p_{l})/2$},\\
w_{l+1}&=&0&\text{on $p=(p_{l}+p_{l+1})/2$},\\
w_{l}&=&w_{l+1}&\text{on $p=p_l$},\\
w_{l,p}-w_{l+1,p}&=&\xi&\text{on $p=p_l$,}
\end{array}
\right.
\end{equation*}
with
\[
\xi:=\varphi_l-v_{l,p}\big|_{p=p_l}+v_{l+1,p}\big|_{p=p_l}\in C^{1+\alpha}(\s).
\]
The arguments presented in the proof of Theorem \ref{T:1} show that the mapping $[\xi\mapsto w_{l}|_{p=p_l}]$ is a Fourier multiplier and it belongs to $\kL(C^{1+\alpha}(\s), C^{2+\alpha}(\s)).$
But then $(w_l,w_{l+1})$ has the same regularity as $(v_l,v_{l+1})$.
Thus, we have shown that $(z_1,\ldots, z_N)\in\Pi_{i=1}^{N}C^{2+\alpha}(\ov\0_i)$ and   the desired Fredholm property follows at once.
The proof of Theorem \ref{MT} follows now   similarly as in the case $N=2$ because all the arguments that we still need are trivial extensions of those presented when $N=2$ (and therefore we omit them).\medskip

\paragraph{\bf The dispersion relation} We end this paper by considering again the case $N=2.$
We determine  an explicit relation, the so-called dispersion relation, between the wave properties: mean depth $d$, the average thickness $d_1$ (resp. $d_2$) of the layer of vorticity $\gamma_1$ (resp. $\gamma_2$), 
the wavelength $L=2\pi/k$,   $k\in\N\setminus\{0\},$  the vorticities $\gamma_1$ and $\gamma_2$, and 
the  relative speed at the crest $c-{\bf u}(0)$ which has to be satisfied in order to have bifurcation from the laminar flow solutions. 
Computing the dispersion relation for the case $N=3$ is also possible, but the computations are much more involved. 

More precisely, we look for conditions on the physical parameters which guarantee that the problem  
\begin{align}\label{K}
  \left\{
\begin{array}{rllll}
(a^3v_k')'-k^2av_k&=&0&\text{$p_0<p<p_1$},\\
(A^3V_k')'-k^2AV_k&=&0&\text{$p_1<p<0$},\\
(g+\sigma k^2)V_k(0)&=&\lambda^{3/2}V_k'(0),\\
v_k(p_1)&=&V_k(p_1),\\
v_k'(p_1)&=&V_k'(p_1),\\
v_k(p_0)&=&0
\end{array}
\right.
 \end{align}
 possesses a nontrivial solution $(v_k,V_k)\in C^{\infty}([p_0,p_1])\times C^{\infty}([p_1,0]).$ 
 Assuming first that $\gamma_1 \gamma_2\neq0,$ we know that the general  solutions $v_k, V_K$ of the first two equations of \eqref{K}
 are given by \eqref{EQ:E}, with constants $\beta,\delta,\theta,\vartheta$ which need to be chosen such that the last four equations of \eqref{K} are also satisfied.
 Thus, we are left with a linear system of four  equations with four unknowns.
 Using algebraic manipulations,  we find that the latter system possesses nontrivial solutions exactly when  the following relation is satisfied
\begin{equation*}
\begin{aligned}
 &\frac{\gamma_2-\gamma_1}{a(p_1)}\left(g+\sigma k^2-\lambda^{1/2}\gamma_2-\lambda k\coth(k\Theta_2)\right)\\
 &=k(\coth(k\Theta_1)+\coth(k\Theta_2))\left(g+\sigma k^2-\lambda^{1/2}\gamma_2-\lambda k\coth(k(\Theta_1+\Theta_2))\right),
\end{aligned}
\end{equation*}
where  $\Theta_1$ and $\Theta_2$ are  given by \eqref{Expr}.
This formula is also true when $\gamma_1\gamma_2=0$ (if $\gamma_1=0$ (resp. $\gamma_2=0$), then $v_k$ (resp. $V_k$) solves an ordinary differential equation with constants coefficients and the computations are easier).
By virtue of the formula \eqref{Expr}, we also have  that $a(p_1)=A(p_1)=\lambda^{1/2}+\gamma_2\Theta_2.$

In order to obtain the desired dispersion relation, we need to give an interpretation to $\Theta_1$ and $\Theta_2.$ 
Therefore, we  recall that the height function $h$ gives the height of a particle above the flat bed,
which implies that
\[
d=\frac{1}{2\pi}\int_\s h(q,0)\, dq\qquad\text{and}\qquad d_1=\frac{1}{2\pi}\int_\s h(q,p_1)\, dq,
\]
where $d_1$ denotes the average height of the fluid layer bounded from below by the flat bed and from above by the interface $\eta_1:=h(\cdot,p_1)-d$ separating the two currents of different vorticities $\gamma_1$ and $\gamma_2$. 
The constant $d_2:=d-d_1>0$ is the average thickness of the fluid layer of vorticity $\gamma_2.$
It follows then readily from the formula \eqref{E:LFS} that in fact
\[
\Theta_1=d_1\qquad\text{and}\qquad\Theta_2=d_2.
\]
Consequently, we obtain the following dispersion relation 
\begin{equation}\label{DRa}
\begin{aligned}
 &\frac{\gamma_2-\gamma_1}{\lambda^{1/2}+\gamma_2d_2}\left(g+\sigma k^2-\lambda^{1/2}\gamma_2-\lambda k\coth(kd_2)\right)\\
 &=k(\coth(kd_1)+\coth(kd_2))\left(g+\sigma k^2-\lambda^{1/2}\gamma_2-\lambda k\coth(kd)\right),
\end{aligned}
\end{equation}
or equivalently
\begin{equation}\label{DRRR}
\begin{aligned}
 \lambda^{3/2}&+\frac{1}{k}\left[\gamma_2\left(d_2+\frac{\sinh(kd_2)\cosh(kd_1)}{\cosh(kd)}\right)+\gamma_1\frac{\sinh(kd_1)\cosh(kd_2)}{\cosh(kd)}\right]\lambda\\
 &+\tanh(kd)\left[\frac{\gamma_2^2d_2-(g+\sigma k^2)}{k}+\gamma_2(\gamma_1-\gamma_2)\frac{\sinh(kd_1)\sinh(kd_2)}{k^2\sinh(kd)}\right]\lambda^{1/2}\\
 &+\frac{(g+\sigma k^2)\tanh(kd)}{k^2}\left[\frac{(\gamma_2-\gamma_1)\sinh(kd_1)\sinh(kd_2)}{\sinh(kd)}-\gamma_2d_2 k\right]=0,
\end{aligned}
\end{equation}
where $\lambda $ is related to the relative speed at the crest by the equation $\sqrt{\lambda}=c-{\bf u}(0)$. 
In general, there is local bifurcation if the equation \eqref{DRRR} has a positive root for $\sqrt{\lambda}.$

It is easy to infer from \eqref{DRa} that if $\gamma_1=\gamma_2=:\gamma,$ then we obtain the dispersion relation for capillary-gravity waves on a linearly sheared current, cf. \cite{W06b},
as we have
\[
c-{\bf u}(0)=-\frac{\gamma}{2}\frac{\tanh(kd)}{k}+\sqrt{\frac{\gamma^2}{4}\left(\frac{\tanh(kd)}{k}\right)^2+(g+\sigma k^2)\frac{\tanh(k)}{k}}.
\]
Moreover, setting $\sigma=0$ in \eqref{DRa} we obtain the dispersion relation for gravity water waves obtained in \cite{CS11} when $\gamma_1=0,$ respectively in \cite{AC12a} for $\gamma_2=0$ (see also \cite{DH13}).
In the case when $\gamma_1=\gamma_2=0,$ we obtain from \eqref{DRa} the dispersion relation for irrotational capillary-gravity water waves
\[
c-{\bf u}(0)=\sqrt{(g+\sigma k^2)\frac{\tanh(k)}{k}},
\]
cf. \cite{Joh97, Li78}.

\bibliographystyle{abbrv}
\bibliography{BC_WS}
\end{document}